\documentclass[a4paper]{article}
\usepackage{amsmath,amssymb}
\usepackage{pst-all}
\usepackage{amsthm}
\newtheorem{theorem}{Theorem}
\newtheorem{proposition}[theorem]{Proposition}

\newtheorem{lemma}[theorem]{Lemma}
\theoremstyle{definition}
\newtheorem{definition}[theorem]{Definition}
\newtheorem{problem}[theorem]{Problem}

\DeclareMathOperator{\dist}{dist}

\title{$k$-$L(2,1)$-Labelling for Planar Graphs is NP-Complete for $k\geq 4$}

\author{Nicole Eggemann\thanks{Brunel University, Kingston Lane, Uxbridge, UB8 3PH, UK. Supported by the EC Marie Curie programme NET-ACE (MEST-CT-2004-6724). {\tt Nicole.Eggemann@brunel.ac.uk}}\ ,
Fr\'ed\'eric Havet\thanks{projet Mascotte, I3S(CNRS and University of  Nice-Sophia Antipolis) and  INRIA,
2004 Route des Lucioles, BP 93, 06902 Sophia-Antipolis Cedex, France.
Partially supported by the european project {\sc FET - Aeolus}. {\tt Frederic.Havet@sophia.inria.fr.}}\ \ and Steven D. Noble\thanks{Brunel University, Kingston Lane, Uxbridge, UB8 3PH, UK. Partially supported by the Heilbronn Institute for Mathematical Research, Bristol, UK. {\tt Steven.Noble@brunel.ac.uk}}}

\begin{document}
\maketitle

\begin{abstract}
A mapping from the vertex set of a graph $G=(V,E)$ into an interval of integers
$\{0, \dots ,k\}$ is an $L(2,1)$-labelling of $G$ of span $k$ if any two
adjacent vertices are mapped onto integers that are at least 2
apart, and every two vertices at distance 2 are mapped onto
distinct integers. It is known that for any fixed $k\ge 4$, deciding
the existence of such a labelling is an NP-complete problem while it is polynomial for $k\leq 3$.
For even $k\geq 8$, it remains NP-complete when restricted to planar graphs.
In this paper, we show that it remains NP-complete for any $k \ge 4$ by reduction from Planar Cubic Two-Colourable Perfect Matching.
Schaefer stated without proof that Planar Cubic Two-Colourable Perfect Matching is NP-complete.
In this paper we give a proof of this.
\end{abstract}

\section{Introduction}
The Frequency Assignment Problem requires the assignment of frequencies to
radio transmitters in a broadcasting network with the aim of avoiding
undesired interference and minimising bandwidth. One of the longstanding graph theoretical models of this problem
is the notion of distance constrained
labelling of graphs. An {\it $L(2,1)$-labelling} of a graph $G$ is a
mapping from the vertex set of $G$ into the nonnegative integers such
that the labels assigned to adjacent vertices  differ by at least 2,
and labels assigned to vertices at distance 2 are different. The
{\it span} of such a labelling is the maximum label used. In this
model, the vertices of $G$ represent the transmitters and the edges
of $G$ express which pairs of transmitters are too close to each
other so that undesired interference may occur, even if the
frequencies assigned to them differ by 1. This model was introduced
by Roberts~\cite{Rob} and since then the concept has been
intensively studied (see the survey articles~\cite{calamoneri:06,yeh:06}). Much of the early research involved determining the optimal labelling of grids and is either folklore or buried in the engineering literature.

The minimum span of an $L(2,1)$-labelling of a graph $G$ is denoted by $\lambda_{2,1}(G)$.
Considerable effort has been spent on trying to resolve a conjecture of Griggs and Yeh~\cite{griggs+yeh} stating that $\lambda_{2,1}(G) \leq \Delta^2$ for graphs with maximum degree $\Delta$. For general graphs, the result has been established for large $\Delta$ by Havet, Reed and Sereni~\cite{havet+reed+sereni}.
For the special case of planar graphs with $\Delta\geq 7$, the conjecture follows from a result of van den Heuvel and McGuiness~\cite{heuvel+mcguiness} and has since been established for planar graphs with $\Delta \ne 3$ by Bella \textit{et al.}~\cite{bella}.
Furthermore, generalising Wegner's conjecture~\cite{wegner:77} on the chromatic number of squares of planar graphs,
it is conjectured that if $G$ is planar then $\lambda_{2,1}(G) \leq \frac{3}{2}\Delta +C$ for some absolute constant $C$. Havet et al.~\cite{havet+heuvel+mcdiarmid+reed:list-colouring-conference-version,havet+heuvel+mcdiarmid+reed:list-colouring-full-version} showed that this conjecture holds asymptotically:  $\lambda_{2,1}(G) \leq \frac{3}{2}\Delta +o(\Delta)$.
For outerplanar graphs with $\Delta \geq 8$, it has been shown that $\lambda_{2,1}(G) \leq \Delta+2$~\cite{calamoneri+petreschi:planar,koller:05,koller:09}.

In their seminal paper, Griggs and Yeh~\cite{griggs+yeh} proved that
determining $\lambda_{2,1}(G)$ is an NP-hard
problem. Fiala, Kloks and Kratochv\'{\i}l \cite{fiala+kloks+kratochvil:fixed-parameter-lambda} proved that deciding
$\lambda_{2,1}(G)\le k$ is NP-complete for every fixed $k\ge 4$.

Since then this problem has been shown to be NP-complete for some very restricted classes of graphs.
For instance, Bodlaender \textit{et al.}~\cite{bodlaender:04} showed that this problem is NP-complete when restricted to bipartite planar graphs if we require $k\ge 8$ and $k$ even.
Later Havet and Thomass\'e~\cite{havet} proved that for any $k\geq 4$,
it remains NP-complete when restricted to a different subclass of bipartite graphs, namely \emph{incidence graphs}, that is, those graphs which may be obtained by making a single subdivision of each edge of a graph.

When the span $k$ is part of the input, the problem
is nontrivial for trees but a polynomial time algorithm based on
bipartite matching was presented by Chang and Kuo in~\cite{chang+kuo}. Since then faster algorithms have been introduced by Hasunuma \textit{et al.} running in time $O(n^{1.75})$~\cite{hasunuma+ishii+ono+uni:fast} and more recently $O(n)$~\cite{hasunuma+ishii+ono+uni:linear} on trees with $n$ vertices.
The problem is still
solvable in polynomial time if the input graph is outerplanar~\cite{koller:05,koller:09}.

Moreover, somewhat
surprisingly, Fiala, Golovach and Kratochv\'{\i}l have shown that the problem becomes NP-complete for series-parallel
graphs~\cite{fiala+golovach+kratochvil:05}, and thus the $L(2,1)$-labelling problem belongs
to a handful of problems known to separate graphs of tree-width 1 and
2 by P/NP-completeness dichotomy.

In this paper we consider the following problem.

\begin{problem}[Planar $k$-$L(2,1)$-Labelling]\label{pro:main}
\mbox{ }\\
Let $k\ge 4$ be fixed. \newline
Instance: A planar graph $G$.\newline
Question: Is there an $L(2,1)$-labelling with span $k$?
\end{problem}
As we mentioned above, Bodlaender \textit{et al.}~\cite{bodlaender:04} showed that this problem is NP-complete if we require $k\ge 8$ and $k$ even.
We have read a suggestion in the literature that the problem is proved to be NP-complete for all $k\ge8$ in~\cite{fotakis+niko+papa:00}.
However this does not seem to be the case. In~\cite{fotakis+niko+papa:00} there is a proof showing that the corresponding problem where $k$ is specified as part of the input is NP-complete. This proof shows that the problem is NP-complete for certain fixed values of $k$. However it is far from clear for which values of $k$ this is true. The same authors also show in~\cite{fotakis+niko+papa:05} that the problem is NP-complete for $k=8$.

In this paper we first prove that Planar Cubic Two-Colourable Perfect Matching, which we define in the next section, is NP-Complete. This result was first stated by Schaefer~\cite{schaefer:78} but without proof.
In the second part of this paper we use this result in order to show that Problem~\ref{pro:main} is NP-complete.

\section{Preliminary results}\label{seccol}

The starting problem for our reductions is Not-All-Equal 3SAT, which is defined as follows~\cite{schaefer:78}.

\begin{definition}[\textsc{Not-All-Equal 3SAT}]
\mbox{ }\\
Instance: A set of clauses each having three literals.\newline
Question: Can the literals be assigned value true or false so that each clause has at least one true and at least one false literal?
\end{definition}

In~\cite{schaefer:78}, it is shown that this problem is NP-complete.

Our reduction involves an intermediate problem concerning a special form of two-colouring. In this section we define the intermediate
problem and show that it is NP-complete.
When $k=4$ or $k=5$, the final stage of our reduction is similar to the reduction in~\cite{fiala+kloks+kratochvil:fixed-parameter-lambda}. However we cannot use induction for higher values of $k$ in contrast with the situation in~\cite{fiala+kloks+kratochvil:fixed-parameter-lambda} and the problem from which the reduction starts in~\cite{fiala+kloks+kratochvil:fixed-parameter-lambda} is not known to be NP-complete for planar graphs. So considerably more work is required.

The following problem is also discussed in~\cite{schaefer:78}.

\begin{problem}[\textsc{Two-Colourable Perfect Matching}]\label{pro:matching}
\mbox{ }\\
Instance: A graph $G$.\newline
Question: Is there a colouring of the vertices of $G$ with colours black and white in which every vertex  has exactly one neighbour of the same colour?
\end{problem}

In~\cite{schaefer:78} it was shown that Two-Colourable Perfect Matching is NP-complete. We are more interested in the case where the input
is restricted to being a planar cubic graph. We call this variant, Planar Cubic Two-Colourable Perfect Matching defined formally as follows~\cite{schaefer:78}.

\begin{problem}[\textsc{Planar Cubic Two-Colourable Perfect Matching}]\label{pro:planarmatching}
\mbox{ }\\
Instance: A planar cubic graph $G$.\newline
Question: Is there a colouring of the vertices of $G$ with colours black and white in which every vertex  has exactly one neighbour of the same colour?
\end{problem}

Schaefer~\cite{schaefer:78} states that this problem is NP-complete but does not give the details of the proof.
We call a colouring as required in Problem~\ref{pro:planarmatching} a \emph{two-coloured perfect matching}.
This section is devoted to the proof of this result, using a reduction from Not-All-Equal 3SAT~\cite{schaefer:78}. As far as we know, no proof of this has ever been published.

We say that a colouring of the vertices of a graph with colours black and white is an \emph{almost two-coloured perfect matching} if every vertex of degree at least two is adjacent to exactly one vertex of the same colour.
We say an edge is \emph{monochromatic} if both end-vertices have the same colour and \emph{dichromatic} if its end-vertices have different colours.

\begin{figure}[ht]
\psset{xunit=0.5cm,yunit=0.5cm,runit=0.5cm}
\begin{center}
\begin{pspicture}(10,7)
\cnode*(5,7){0.15}{a}
\rput(5.375,7.175){$a$}
\cnode*(5,6){0.15}{b}
\rput(5.375,6.275){$b$}
\ncline{-}{a}{b}
\cnode*(3,5.3){0.15}{c}
\rput(3.1,5.7){$c$}
\cnode*(7,5.3){0.15}{d}
\rput(7.1,5.8){$d$}
\ncline{-}{b}{c}
\ncline{-}{b}{d}
\cnode*(2.5,4.5){0.15}{e}
\rput(2.1,4.5){$e$}
\cnode*(7.5,4.5){0.15}{h}
\rput(7.9,4.5){$h$}
\cnode*(3.5,4.5){0.15}{f}
\rput(3.9,4.5){$f$}
\cnode*(6.5,4.5){0.15}{g}
\rput(6.1,4.5){$g$}
\ncline{-}{c}{e}
\ncline{-}{c}{f}
\ncline{-}{d}{g}
\ncline{-}{d}{h}
\cnode*(2.5,2.5){0.15}{i}
\rput(2.1,2.9){$i$}
\cnode*(7.5,2.5){0.15}{l}
\rput(7.9,2.9){$l$}
\cnode*(3.5,2.5){0.15}{j}
\rput(3.9,2.9){$j$}
\cnode*(6.5,2.5){0.15}{k}
\rput(6.1,2.9){$k$}
\ncline{-}{e}{i}
\ncline{-}{f}{j}
\ncline{-}{g}{k}
\ncline{-}{h}{l}
\ncline{-}{e}{f}
\ncline{-}{g}{h}
\cnode*(3,1.5){0.15}{o}
\rput(2.55,1.5){$o$}
\cnode*(7,1.5){0.15}{p}
\rput(6.55,1.5){$p$}
\cnode*(3,0.5){0.15}{q}
\rput(2.55,0.5){$q$}
\cnode*(7,0.5){0.15}{r}
\rput(6.55,0.5){$r$}
\cnode*(1.5,2.5){0.15}{m}
\rput(1.4,2.9){$m$}
\cnode*(8.5,2.5){0.15}{n}
\rput(8.6,2.9){$n$}
\ncline{-}{i}{o}
\ncline{-}{j}{o}
\ncline{-}{o}{q}
\ncline{-}{p}{l}
\ncline{-}{k}{p}
\ncline{-}{p}{r}
\ncline{-}{m}{i}
\ncline{-}{n}{l}
\ncline{-}{j}{k}
\end{pspicture}
\end{center}
\caption{Planar graph $H$.}\label{fig1}
\end{figure}
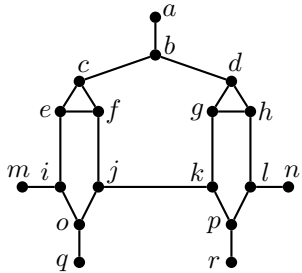

Let $H$ be the planar graph depicted in  Fig.~\ref{fig1}.
$H$ plays a key role in showing that Problem~\ref{pro:planarmatching} is NP-complete.
We need the following lemma.

\begin{lemma}\label{le:proph}
Any almost two-coloured perfect matching of $H$ has the following properties.
\begin{itemize}
\item Exactly one of the edges $ab,mi,ln$ is monochromatic.
\item Vertices $b,i,l$ receive the same colour.
\item Vertices $o,p,q,r$ receive the other colour to $b,i,l$.
\end{itemize}
\end{lemma}
\begin{proof}
Consider the triangles on the vertices $c,e,f$ and $g,h,d$. In order to obtain an almost two-coloured perfect matching exactly one of the edges $ce, ef$ and $cf$ must be monochromatic. The same is true for the triangle on the vertices $g,h,d$.
Now consider the subgraph of $H$ induced by the vertices $c,e,f,i,j,m,k,o,q$. In Fig.~\ref{fig3} three of the six almost two-coloured perfect matchings of this subgraph are depicted with monochromatic edges shown by heavy lines. The other three two-coloured perfect matchings are obtained by interchanging the colours. It follows that $oq$ and $pr$ must be monochromatic.

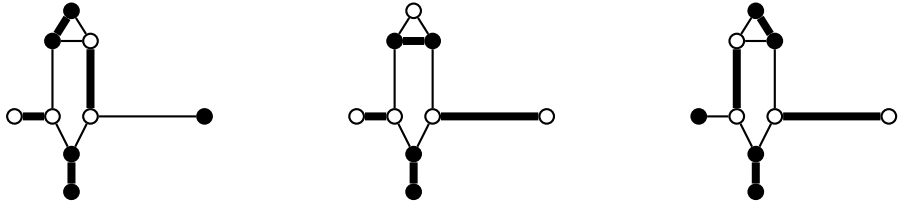
\begin{figure}[ht]
\begin{center}
\psset{xunit=0.5cm,yunit=0.5cm,runit=0.5cm}

\begin{pspicture}(24,6.5)

\cnode*(2,5.3){0.22}{c}

\cnode*(1.5,4.5){0.22}{e}

\cnode(2.5,4.5){0.22}{f}

\ncline[linewidth=3pt]{-}{c}{e}
\ncline{-}{c}{f}
\ncline{-}{d}{g}

\cnode(1.5,2.5){0.22}{i}

\cnode(2.5,2.5){0.22}{j}
\cnode*(5.5,2.5){0.22}{k}
\ncline{-}{e}{i}
\ncline[linewidth=3pt]{-}{f}{j}
\ncline{-}{e}{f}
\cnode*(2,1.5){0.22}{o}
\cnode*(2,0.5){0.22}{q}
\cnode(0.5,2.5){0.22}{m}
\ncline{-}{i}{o}
\ncline{-}{j}{o}
\ncline[linewidth=3pt]{-}{o}{q}
\ncline[linewidth=3pt]{-}{m}{i}
\ncline{-}{j}{k}
\cnode(11,5.3){0.22}{c}
\cnode*(10.5,4.5){0.22}{e}
\cnode*(11.5,4.5){0.22}{f}
\ncline{-}{c}{e}
\ncline{-}{c}{f}
\cnode(10.5,2.5){0.22}{i}
\cnode(11.5,2.5){0.22}{j}
\cnode(14.5,2.5){0.22}{k}
\ncline{-}{e}{i}
\ncline{-}{f}{j}
\ncline[linewidth=3pt]{-}{e}{f}
\cnode*(11,1.5){0.22}{o}
\cnode*(11,0.5){0.22}{q}
\cnode(9.5,2.5){0.22}{m}
\ncline{-}{i}{o}
\ncline{-}{j}{o}
\ncline[linewidth=3pt]{-}{o}{q}
\ncline[linewidth=3pt]{-}{m}{i}
\ncline[linewidth=3pt]{-}{j}{k}

\cnode*(20,5.3){0.22}{c}
\cnode(19.5,4.5){0.22}{e}
\cnode*(20.5,4.5){0.22}{f}
\ncline{-}{c}{e}
\ncline[linewidth=3pt]{-}{c}{f}
\cnode(19.5,2.5){0.22}{i}
\cnode(20.5,2.5){0.22}{j}
\cnode(23.5,2.5){0.22}{k}
\ncline[linewidth=3pt]{-}{e}{i}
\ncline{-}{f}{j}
\ncline{-}{e}{f}
\cnode*(20,1.5){0.22}{o}
\cnode*(20,0.5){0.22}{q}
\cnode*(18.5,2.5){0.22}{m}
\ncline{-}{i}{o}
\ncline{-}{j}{o}
\ncline[linewidth=3pt]{-}{o}{q}
\ncline{-}{m}{i}
\ncline[linewidth=3pt]{-}{j}{k}
\end{pspicture}
\end{center}
\caption{Three almost two-coloured perfect matchings of the subgraph of $H$ induced by the vertices $c,e,f,m,i,j,k,o,q$.}\label{fig3}
\end{figure}

By symmetry the same applies to the subgraph of $H$ induced by the vertices $d,g,h$, $j,k,l,n,p,r$. Considering which pairs of these almost two-coloured
perfect matchings are compatible and extend to an almost two-coloured perfect matching of $H$ shows that there are only six possibilities.
In Fig.~\ref{fig2} three possible almost two-coloured perfect matchings are depicted.
The only other possible almost two-coloured perfect matchings are obtained by interchanging the two colours.
Clearly these all have the properties described in the lemma.

\begin{figure}[ht]
\psset{xunit=0.5cm,yunit=0.5cm,runit=0.5cm}
\begin{center}
\begin{pspicture}(1,0)(25,7)
\cnode*(5,7){0.22}{a}
\cnode*(5,6){0.22}{b}
\ncline[linewidth=3pt]{-}{a}{b}
\cnode(3,5.3){0.22}{c}
\cnode(7,5.3){0.22}{d}
\ncline{-}{b}{c}
\ncline{-}{b}{d}
\cnode*(2.5,4.5){0.22}{e}
\cnode*(7.5,4.5){0.22}{h}
\cnode(3.5,4.5){0.22}{f}
\cnode(6.5,4.5){0.22}{g}
\ncline{-}{c}{e}
\ncline[linewidth=3pt]{-}{c}{f}
\ncline[linewidth=3pt]{-}{d}{g}
\ncline{-}{d}{h}
\cnode*(2.5,2.5){0.22}{i}
\cnode*(7.5,2.5){0.22}{l}
\cnode*(3.5,2.5){0.22}{j}
\cnode*(6.5,2.5){0.22}{k}
\ncline[linewidth=3pt]{-}{e}{i}
\ncline{-}{f}{j}
\ncline{-}{g}{k}
\ncline[linewidth=3pt]{-}{h}{l}
\ncline{-}{e}{f}
\ncline{-}{g}{h}
\cnode(3,1.5){0.22}{o}
\cnode(7,1.5){0.22}{p}
\cnode(3,0.5){0.22}{q}
\cnode(7,0.5){0.22}{r}
\cnode(1.5,2.5){0.22}{m}
\cnode(8.5,2.5){0.22}{n}
\ncline{-}{i}{o}
\ncline{-}{j}{o}
\ncline[linewidth=3pt]{-}{o}{q}
\ncline{-}{p}{l}
\ncline{-}{k}{p}
\ncline[linewidth=3pt]{-}{p}{r}
\ncline{-}{m}{i}
\ncline{-}{n}{l}
\ncline[linewidth=3pt]{-}{j}{k}

\cnode(13,7){0.22}{a}
\cnode*(13,6){0.22}{b}
\ncline{-}{a}{b}
\cnode*(11,5.3){0.22}{c}
\cnode(15,5.3){0.22}{d}
\ncline[linewidth=3pt]{-}{b}{c}
\ncline{-}{b}{d}
\cnode(10.5,4.5){0.22}{e}
\cnode*(15.5,4.5){0.22}{h}
\cnode(11.5,4.5){0.22}{f}
\cnode(14.5,4.5){0.22}{g}
\ncline{-}{c}{e}
\ncline{-}{c}{f}
\ncline[linewidth=3pt]{-}{d}{g}
\ncline{-}{d}{h}
\cnode*(10.5,2.5){0.22}{i}
\cnode*(15.5,2.5){0.22}{l}
\cnode*(11.5,2.5){0.22}{j}
\cnode*(14.5,2.5){0.22}{k}
\ncline{-}{e}{i}
\ncline{-}{f}{j}
\ncline{-}{g}{k}
\ncline[linewidth=3pt]{-}{h}{l}
\ncline[linewidth=3pt]{-}{e}{f}
\ncline{-}{g}{h}
\cnode(11,1.5){0.22}{o}
\cnode(15,1.5){0.22}{p}
\cnode(11,0.5){0.22}{q}
\cnode(15,0.5){0.22}{r}
\cnode*(9.5,2.5){0.22}{m}
\cnode(16.5,2.5){0.22}{n}
\ncline{-}{i}{o}
\ncline{-}{j}{o}
\ncline[linewidth=3pt]{-}{o}{q}
\ncline{-}{p}{l}
\ncline{-}{k}{p}
\ncline[linewidth=3pt]{-}{p}{r}
\ncline[linewidth=3pt]{-}{m}{i}
\ncline{-}{n}{l}
\ncline[linewidth=3pt]{-}{j}{k}

\cnode(21,7){0.22}{a}
\cnode*(21,6){0.22}{b}
\ncline{-}{a}{b}
\cnode(19,5.3){0.22}{c}
\cnode*(23,5.3){0.22}{d}
\ncline{-}{b}{c}
\ncline[linewidth=3pt]{-}{b}{d}
\cnode*(18.5,4.5){0.22}{e}
\cnode(23.5,4.5){0.22}{h}
\cnode(19.5,4.5){0.22}{f}
\cnode(22.5,4.5){0.22}{g}
\ncline{-}{c}{e}
\ncline[linewidth=3pt]{-}{c}{f}
\ncline{-}{d}{g}
\ncline{-}{d}{h}
\cnode*(18.5,2.5){0.22}{i}
\cnode*(23.5,2.5){0.22}{l}
\cnode*(19.5,2.5){0.22}{j}
\cnode*(22.5,2.5){0.22}{k}
\ncline[linewidth=3pt]{-}{e}{i}
\ncline{-}{f}{j}
\ncline{-}{g}{k}
\ncline{-}{h}{l}
\ncline{-}{e}{f}
\ncline[linewidth=3pt]{-}{g}{h}
\cnode(19,1.5){0.22}{o}
\cnode(23,1.5){0.22}{p}
\cnode(19,0.5){0.22}{q}
\cnode(23,0.5){0.22}{r}
\cnode(17.5,2.5){0.22}{m}
\cnode*(24.5,2.5){0.22}{n}
\ncline{-}{i}{o}
\ncline{-}{j}{o}
\ncline[linewidth=3pt]{-}{o}{q}
\ncline{-}{p}{l}
\ncline{-}{k}{p}
\ncline[linewidth=3pt]{-}{p}{r}
\ncline{-}{m}{i}
\ncline[linewidth=3pt]{-}{n}{l}
\ncline[linewidth=3pt]{-}{j}{k}
\end{pspicture}
\end{center}
\caption{Almost two-coloured perfect matchings of $H$.}\label{fig2}
\end{figure}
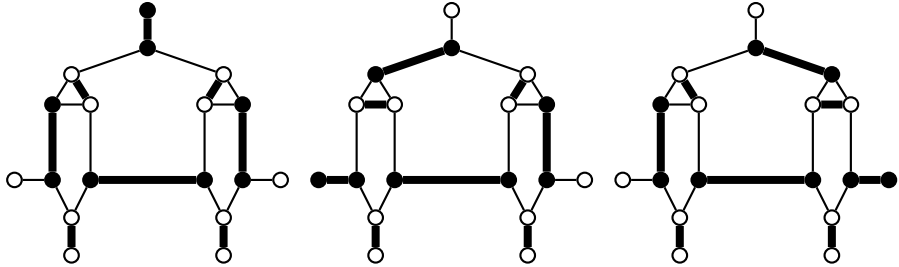
\end{proof}

We define what we call the \emph{clause gadget graph} $K$ as follows, see Fig.~\ref{fig13}. Take three copies of $H$, namely $H_1$, $H_2$ and $H_3$.
We label the vertices by adding the subscript $i \in \{1,2,3\}$ to the corresponding label of $H$. Now identify $a_1,a_2,a_3$ into a single vertex $a$,
remove vertices $m_1,m_2,m_3,n_1,n_2,n_3$ and their incident edges and replace them with edges $l_1i_2,l_2i_3,l_3i_1$. Notice that $K$ is planar
and every vertex has degree three, except for $q_1,q_2,q_3$ and $r_1,r_2,r_3$.

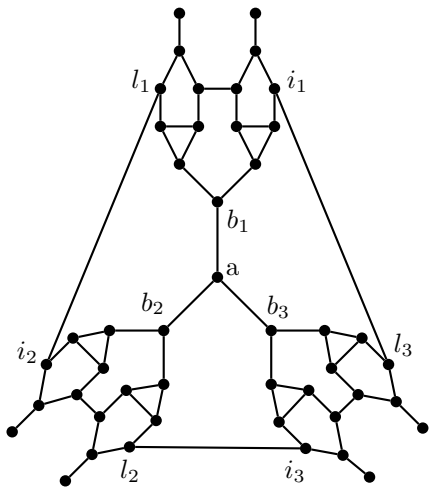
\begin{figure}[ht]
\psset{xunit=1cm,yunit=1cm,runit=1cm}
\begin{center}
\begin{pspicture}(7,7.5)
\cnode*(3.5,3.5){0.075}{a}
\rput(3.7,3.6){a}
\cnode*(3.5,4.5){0.075}{b}
\rput(3.775,4.25){$b_1$}
\cnode*(2.793,2.793){0.075}{c}
\rput(2.65,3.15){$b_2$}
\cnode*(4.207,2.793){0.075}{d}
\rput(4.3,3.1){$b_3$}
\ncline{-}{a}{b}
\ncline{-}{a}{c}
\ncline{-}{a}{d}
\cnode*(3,5){0.075}{e}
\cnode*(4,5){0.075}{f}
\ncline{-}{b}{e}
\ncline{-}{b}{f}
\cnode*(2.75,5.5){0.075}{g}
\cnode*(3.25,5.5){0.075}{h}
\cnode*(3.75,5.5){0.075}{i}
\cnode*(4.25,5.5){0.075}{j}
\ncline{-}{g}{e}
\ncline{-}{e}{h}
\ncline{-}{i}{f}
\ncline{-}{j}{f}
\ncline{-}{i}{j}
\ncline{-}{h}{g}
\cnode*(2.75,6){0.075}{k}
\rput(2.5,6.1){$l_1$}
\cnode*(3.25,6){0.075}{l}
\cnode*(3.75,6){0.075}{m}
\cnode*(4.25,6){0.075}{n}
\rput(4.55,6.1){$i_1$}
\ncline{-}{g}{k}
\ncline{-}{l}{h}
\ncline{-}{i}{m}
\ncline{-}{j}{n}
\ncline{-}{m}{l}
\cnode*(3,6.5){0.075}{o}
\cnode*(4,6.5){0.075}{p}
\ncline{-}{o}{k}
\ncline{-}{l}{o}
\ncline{-}{p}{m}
\ncline{-}{p}{n}
\cnode*(3,7){0.075}{q}
\cnode*(4,7){0.075}{r}
\ncline{-}{p}{r}
\ncline{-}{q}{o}
\cnode*(4.91,2.793){0.075}{d1}
\cnode*(4.207,2.08){0.075}{d2}
\ncline{-}{d}{d1}
\ncline{-}{d}{d2}
\cnode*(5.4,2.693){0.075}{d3}
\cnode*(5,2.293){0.075}{d4}
\ncline{-}{d3}{d1}
\ncline{-}{d4}{d1}
\ncline{-}{d3}{d4}
\cnode*(4.7,2){0.075}{d5}
\cnode*(4.31,1.58){0.075}{d6}
\ncline{-}{d5}{d2}
\ncline{-}{d6}{d2}
\ncline{-}{d5}{d6}
\cnode*(5.75,2.343){0.075}{d7}
\rput(5.95,2.6){$l_3$}
\cnode*(5.35,1.943){0.075}{d8}
\ncline{-}{d3}{d7}
\ncline{-}{d4}{d8}
\cnode*(5.85,1.85){0.075}{d9}
\ncline{-}{d9}{d7}
\ncline{-}{d9}{d8}
\cnode*(5.05,1.65){0.075}{d10}
\cnode*(4.66,1.23){0.075}{d11}
\rput(4.52,0.93){$i_3$}
\ncline{-}{d10}{d5}
\ncline{-}{d11}{d6}
\cnode*(5.15,1.14){0.075}{d12}
\ncline{-}{d10}{d12}
\ncline{-}{d11}{d12}
\cnode*(5.5,0.85){0.075}{d13}
\ncline{-}{d13}{d12}
\cnode*(6.2,1.5){0.075}{d14}
\ncline{-}{d9}{d14}
\ncline{-}{d10}{d8}
\cnode*(2.08,2.793){0.075}{c1}
\cnode*(2.793,2.08){0.075}{c2}
\ncline{-}{c}{c1}
\ncline{-}{c}{c2}
\cnode*(1.6,2.693){0.075}{c11}
\cnode*(2,2.293){0.075}{c12}
\ncline{-}{c11}{c1}
\ncline{-}{c12}{c1}
\ncline{-}{c12}{c11}
\cnode*(1.25,2.343){0.075}{c13}
\rput(1,2.5){$i_2$}
\cnode*(1.65,1.943){0.075}{c14}
\ncline{-}{c11}{c13}
\ncline{-}{c12}{c14}
\cnode*(1.15,1.8){0.075}{c15}
\ncline{-}{c15}{c13}
\ncline{-}{c15}{c14}
\cnode*(0.8,1.45){0.075}{c16}
\ncline{-}{c15}{c16}
\cnode*(2.3,2){0.075}{c21}
\cnode*(2.693,1.6){0.075}{c22}
\ncline{-}{c21}{c2}
\ncline{-}{c22}{c2}
\ncline{-}{c22}{c21}
\cnode*(1.95,1.65){0.075}{c23}
\cnode*(2.343,1.25){0.075}{c24}
\rput(2.35,0.93){$l_2$}
\ncline{-}{c21}{c23}
\ncline{-}{c22}{c24}
\cnode*(1.85,1.15){0.075}{c25}
\ncline{-}{c25}{c23}
\ncline{-}{c25}{c24}
\cnode*(1.5,0.8){0.075}{c26}
\ncline{-}{c25}{c26}
\ncline{-}{c23}{c14}
\ncline{-}{c24}{d11}
\ncline{-}{c13}{k}
\ncline{-}{d7}{n}
\end{pspicture}
\end{center}
\caption{Planar clause gadget $K$.}\label{fig13}
\end{figure}

\begin{lemma}\label{le:colouringc}
A two-colouring of $\bigcup^3_{t=1}\{o_t,q_t,p_t,r_t\}\cup \{a\}$ may be extended to an almost two-coloured perfect matching of $K$ if and only if
\begin{itemize}
\item For each $t=1,2,3$, $o_tq_t$, $p_tr_t$ are monochromatic and $o_t,p_t,q_t,r_t$ all receive the same colour.
\item For exactly two values of $t=1,2,3$, the vertices $o_t,p_t,q_t,r_t$ receive the same colour as $a$.
\end{itemize}
\end{lemma}

\begin{proof}
We first show that any almost two-coloured perfect matching of $K$ must have the two properties in the lemma.

The first property is an immediate consequence of Lemma~\ref{le:proph}.

To show that the second property holds, recall
that exactly one neighbour of $a$ must receive the same colour as $a$. Let $b_{t_1}$ for $1 \le t_1 \le 3$ be this neighbour. Then from Lemma~\ref{le:proph} we know that $b_{t_1}$ must have the opposite colour to $o_{t_1},p_{t_1},q_{t_1},r_{t_1}$. Since the other neighbours of $a$, namely $b_{t_2}$ and $b_{t_3}$ for $t_2,t_3 \in \{1,2,3\} \backslash \{t_1 \}$, receive the opposite colour to $a$, the vertices $o_{t_2},p_{t_2},q_{t_2},r_{t_2},o_{t_3},p_{t_3},q_{t_3},r_{t_3}$ must receive the same colour as $a$.

Now we show that any two-colouring of $\bigcup^3_{t=1}\{o_t,q_t,p_t,r_t,b_t \}\cup \{a\}$ satisfying the conditions of the lemma may be extended to an almost two-coloured perfect matching of $K$. Suppose without loss of generality that $a$ is coloured black and $o_1,p_1,q_1,r_1$ are coloured white. Then colour $l_1,i_1$ black and $l_2,i_2,l_3,i_3$ white. This colouring may be extended to an almost two-coloured perfect matching using the colourings of Fig.~\ref{fig2} and the colourings obtained from those in Fig.~\ref{fig2} by interchanging the colours.
\end{proof}

We now move a step towards the main result of this section with the following proposition.
\begin{proposition}\label{prop:1}
Problem~\ref{pro:matching} is NP-complete if the input is restricted to cubic graphs.
\end{proposition}

\begin{proof}
Given an instance of Not-All-Equal 3SAT with clauses $C_1,...,C_m$, construct a graph as follows. For every clause $C$ take a copy of the clause gadget graph $K(C)$ and do the following. Suppose without loss of generality that $C$ has literals $x_1,x_2,x_3$. Label the two vertices of degree one of the subgraph $H_i$ of $K(C)$ and their neighbours in $H_i$ with $x_i$.

Now for each literal $x$ do the following. Suppose that literal $x$ appears in clauses $C_{i_1},...,C_{i_k}$. (If $x$ appears twice or three times in a clause $C_r$ then add $C_r$ twice or three times to this list.) For every $j=1,...,k-1$ remove either one of the
vertices of degree one labelled $x$ from $K(C_{i_j})$ and from $K(C_{i_{j+1}})$ leaving two half-edges. Now identify these two half edges to form an edge joining
$K(C_{i_j})$ and $K(C_{i_{j+1}})$. Finally do the same thing with the remaining two edges labelled $x$ in $C_{i_k}$ and $C_{i_1}$. We call the graph  obtained $G$.

Now suppose that there is a solution $S$ of the instance of Not-All-Equal 3SAT given.
For all literals $x$ in Not-All-Equal 3SAT, colour all vertices in $G$ that are labelled $x$ with colour white if $S(x)$ is true and black if $S(x)$ is false. We now show that this colouring can be extended to a two-coloured perfect matching of $G$. First note that all edges joining copies of $K$ are monochromatic since their end-vertices are labelled with the same literal. In the next step colour the vertex $a$ in each copy of $K$ so that it has the same colour as the vertices $o_t,p_t$ for exactly two values of $t=1,2,3$. This is possible because for $t=1,2,3$, $o_t,p_t$ are labelled with literals $x_1,x_2,x_3$ which cannot have all the same value as they belong to one clause.
By Lemma~\ref{le:colouringc} we can extend the colouring of each copy of $K$ to an almost two-coloured perfect matching of $K$ which yields a two-coloured perfect matching of $G$.

Now suppose there is a two-coloured perfect matching of $G$. By Lemma~\ref{le:colouringc} the edges joining the copies of $K$ must be monochromatic. All vertices in $G$ labelled with the same literal therefore must have the same colour and in each copy of $K$, for exactly two values of $t=1,2,3$, the vertices $o_t,p_t$ receive the same colour. It follows that if we assign to each literal $x$ the value true if it is the label of white vertices and false it is the label of black vertices then we obtain a solution to Not-All-Equal 3SAT.
\end{proof}

We call the edges joining copies of $K$ \emph{identifying edges}.

In order to prove that Problem~\ref{pro:planarmatching} is NP-complete we still need to deal with edges that cross.
For this reason we define the \emph{uncrossing gadget} $U$ to be the graph depicted in Fig.~\ref{fig80}.
\begin{figure}[ht]
\psset{xunit=0.7cm,yunit=0.7cm,runit=0.7cm,labelsep=1.8pt}
\begin{center}
\begin{pspicture}(18,9)
\cnode*(4,0.5){0.1}{z4}
\uput[315](4,0.5){\small $z_4$}
\cnode*(4,1.5){0.1}{z3}
\uput[315](4,1.5){\small $z_3$}
\cnode*(3,2.5){0.1}{u}
\uput[225](3,2.5){\small $u$}
\cnode*(5,2.5){0.1}{v}
\uput[315](5,2.5){\small $x$}
\ncline{-}{z4}{z3}
\ncline{-}{z3}{u}
\ncline{-}{z3}{v}
\ncline{-}{u}{v}
\cnode*(3,3.5){0.1}{n}
\uput[225](3,3.5){\small $n$}
\cnode*(5,3.5){0.1}{q}
\uput[315](5,3.5){\small $q$}
\ncline{-}{u}{n}
\ncline{-}{q}{v}
\cnode*(2.5,4){0.1}{l}
\uput[225](2.5,4){\small $l$}
\cnode*(3.5,4){0.1}{m}
\uput[315](3.5,4){\small $m$}
\cnode*(5.5,4){0.1}{r}
\uput[315](5.5,4){\small $r$}
\cnode*(4.5,4){0.1}{p}
\uput[225](4.5,4){\small $p$}
\ncline{-}{n}{l}
\ncline{-}{n}{m}
\ncline{-}{p}{q}
\ncline{-}{r}{q}
\ncline{-}{p}{m}
\cnode*(1.5,5){0.1}{v}
\uput[225](1.5,5){\small $v$}
\ncline{-}{v}{l}
\cnode*(0.5,5){0.1}{w}
\uput[225](0.5,5){\small $w$}
\ncline{-}{v}{w}
\cnode*(3,4.5){0.1}{k}
\uput[135](3,4.5){\small $k$}
\cnode*(5,4.5){0.1}{o}
\uput[135](5,4.5){\small $o$}
\ncline{-}{k}{l}
\ncline{-}{k}{m}
\ncline{-}{p}{o}
\ncline{-}{r}{o}
\cnode*(3,5.5){0.1}{f}
\uput[315](3,5.5){\small $f$}
\cnode*(5,5.5){0.1}{i}
\uput[315](5,5.5){\small $i$}
\ncline{-}{f}{k}
\ncline{-}{o}{i}
\cnode*(2.5,6){0.1}{e}
\uput[135](2.5,6){\small $e$}
\cnode*(3.5,6){0.1}{d}
\uput[45](3.5,6){\small $d$}
\cnode*(5.5,6){0.1}{j}
\uput[45](5.5,6){\small $j$}
\cnode*(4.5,6){0.1}{h}
\uput[135](4.5,6){\small $h$}
\ncline{-}{f}{e}
\ncline{-}{f}{d}
\ncline{-}{h}{i}
\ncline{-}{j}{i}
\ncline{-}{h}{d}
\cnode*(3,6.5){0.1}{c}
\uput[135](3,6.5){\small $c$}
\cnode*(5,6.5){0.1}{g}
\uput[45](5,6.5){\small $g$}
\ncline{-}{c}{e}
\ncline{-}{c}{d}
\ncline{-}{h}{g}
\ncline{-}{j}{g}
\ncline{-}{e}{v}
\cnode*(4,7.5){0.1}{b}
\uput[45](4,7.5){\small $b$}
\cnode*(4,8.5){0.1}{a}
\rput[45](4,8.5){\small $a$}
\ncline{-}{c}{b}
\ncline{-}{a}{b}
\ncline{-}{b}{g}
\cnode*(6.5,6){0.1}{s}
\uput[45](6.5,6){\small $s$}
\cnode*(6.5,4){0.1}{t}
\uput[315](6.5,4){\small $t$}
\ncline{-}{s}{t}
\ncline{-}{s}{j}
\ncline{-}{t}{r}
\cnode*(7.5,5){0.1}{z1}
\uput[45](7.5,5){\small $z_1$}
\cnode*(8.5,5){0.1}{z2}
\uput[45](8.5,5){\small $z_2$}
\ncline{-}{z1}{z2}
\ncline{-}{s}{z1}
\ncline{-}{t}{z1}

\cnode*(13,0.5){0.1}{z4}
\uput[315](13,0.5){\small $\alpha$}
\cnode*(13,1.5){0.1}{z3}
\uput[315](13,1.5){\small $\alpha$}
\cnode*(12,2.5){0.1}{u}
\uput[225](12,2.5){\small $\alpha'$}
\cnode*(14,2.5){0.1}{v}
\uput[315](14,2.5){\small $\alpha$}
\ncline{-}{z4}{z3}
\ncline{-}{z3}{u}
\ncline{-}{z3}{v}
\ncline{-}{u}{v}
\cnode*(12,3.5){0.1}{n}
\uput[225](12,3.5){\small $\alpha$}
\cnode*(14,3.5){0.1}{q}
\uput[315](14,3.5){\small $\alpha$}
\ncline{-}{u}{n}
\ncline{-}{q}{v}
\cnode*(11.5,4){0.1}{l}
\uput[225](11.5,4){\small $\beta'$}
\cnode*(12.5,4){0.1}{m}
\uput[315](12.5,4){\small $\beta$}
\cnode*(14.5,4){0.1}{r}
\uput[315](14.5,4){\small $\beta$}
\cnode*(13.5,4){0.1}{p}
\uput[225](13.5,4){\small $\beta'$}
\ncline{-}{n}{l}
\ncline{-}{n}{m}
\ncline{-}{p}{q}
\ncline{-}{r}{q}
\ncline{-}{p}{m}
\cnode*(10.5,5){0.1}{v}
\uput[225](10.5,5){\small $\beta$}
\ncline{-}{v}{l}
\cnode*(9.5,5){0.1}{w}
\uput[225](9.5,5){\small $\beta$}
\ncline{-}{v}{w}
\cnode*(12,4.5){0.1}{k}
\uput[135](12,4.5){\small $\alpha'$}
\cnode*(14,4.5){0.1}{o}
\uput[135](14,4.5){\small $\alpha'$}
\ncline{-}{k}{l}
\ncline{-}{k}{m}
\ncline{-}{p}{o}
\ncline{-}{r}{o}
\cnode*(12,5.5){0.1}{f}
\uput[315](12,5.5){\small $\alpha$}
\cnode*(14,5.5){0.1}{i}
\uput[315](14,5.5){\small $\alpha$}
\ncline{-}{f}{k}
\ncline{-}{o}{i}
\cnode*(11.5,6){0.1}{e}
\uput[135](11.5,6){\small $\beta'$}
\cnode*(12.5,6){0.1}{d}
\uput[45](12.5,6){\small $\beta$}
\cnode*(14.5,6){0.1}{j}
\uput[45](14.5,6){\small $\beta'$}
\cnode*(13.5,6){0.1}{h}
\uput[135](13.5,6){\small $\beta$}
\ncline{-}{f}{e}
\ncline{-}{f}{d}
\ncline{-}{h}{i}
\ncline{-}{j}{i}
\ncline{-}{h}{d}
\cnode*(12,6.5){0.1}{c}
\uput[135](12,6.5){\small $\alpha'$}
\cnode*(14,6.5){0.1}{g}
\uput[45](14,6.5){\small $\alpha'$}
\ncline{-}{c}{e}
\ncline{-}{c}{d}
\ncline{-}{h}{g}
\ncline{-}{j}{g}

\ncline{-}{e}{v}
\cnode*(13,7.5){0.1}{b}
\uput[45](13,7.5){\small $\alpha$}
\cnode*(13,8.5){0.1}{a}
\rput[45](13,8.5){\small $\alpha$}
\ncline{-}{c}{b}
\ncline{-}{a}{b}
\ncline{-}{b}{g}

\cnode*(15.5,6){0.1}{s}
\uput[45](15.5,6){\small $\beta'$}
\cnode*(15.5,4){0.1}{t}
\uput[315](15.5,4){\small $\beta'$}
\ncline{-}{s}{t}
\ncline{-}{s}{j}
\ncline{-}{t}{r}
\cnode*(16.5,5){0.1}{z1}
\uput[45](16.5,5){\small $\beta$}
\cnode*(17.5,5){0.1}{z2}
\uput[45](17.5,5){\small $\beta$}
\ncline{-}{z1}{z2}
\ncline{-}{s}{z1}
\ncline{-}{t}{z1}
\end{pspicture}
\caption{Uncrossing gadget $U$ and its possible almost two-coloured perfect matchings where $\alpha,\beta \in \{\mbox{black},\mbox{white}\}$ and $\alpha',\beta'$ are the opposite colours to $\alpha$ and $\beta$, respectively.}\label{fig80}
\end{center}
\end{figure}
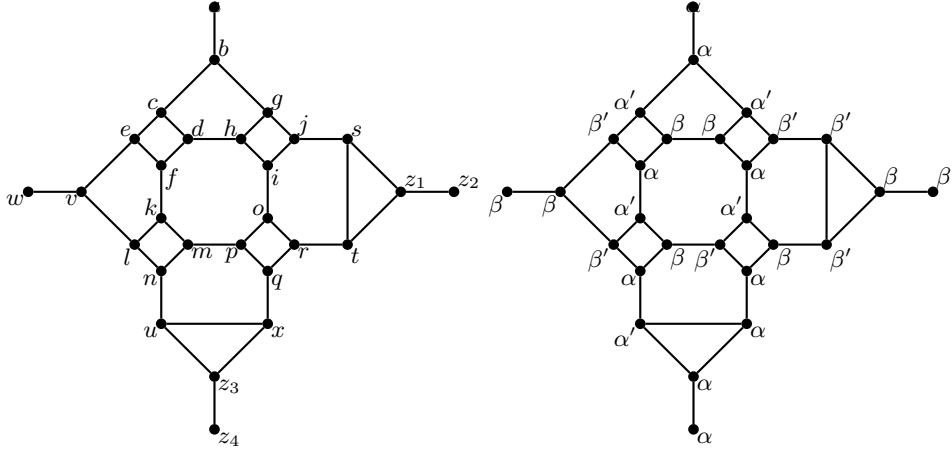

\begin{lemma}\label{le:uncrossgadget}
There exists an almost two-coloured perfect matching of $U$ if and only if $w,v,z_1,z_2$ have the same colour and $a,b,z_3,z_4$ have the same colour.
\end{lemma}
\begin{proof}
Consider the four-cycle on the vertices $c,e,f,d$ in $U$. In a two-coloured perfect matching exactly two of the vertices $c,e,f,d$ must receive
the colour black and the other two must receive the colour white. There are two different ways of colouring them. The first way is that exactly two of the edges in the four-cycle are monochromatic, namely $ce$ and $df$, or $cd$ and $ef$. Then none of the edges $cb,dh,ev,fk$ can be monochromatic.
The other way is that none of the edges in the four-cycle is monochromatic and all of the edges $cb,dh,ev,fk$ are monochromatic.
The analogous thing is true for any four-cycle in $U$.

We now prove that in any almost two-coloured perfect matching the edges $ab,wv,z_1z_2$ and $z_3z_4$ are monochromatic.
Suppose $ab$ is dichromatic. Then precisely one of $bc$ and $bg$ must be monochromatic. Without loss of generality assume $bc$ is monochromatic.
Then the four-cycle on $c,e,f,d$ cannot have a monochromatic edge. It follows that $dh$ must be monochromatic. But then $bg$ must be monochromatic which
is a contradiction. Thus $ab$ and by symmetry $wv$ must be monochromatic.
Now suppose $z_1z_2$ is dichromatic. It follows that precisely one of $sz_1$ and $tz_1$ is monochromatic. Without loss of generality assume $tz_1$ is monochromatic. Then $js$ must be monochromatic. It follows that $io$ must be monochromatic and so must $rt$. This is not possible.
Thus $z_1z_2$ and due to symmetry $z_3z_4$ must be monochromatic.
Hence each of the four-cycles $(c,d,f,e),(g,h,i,j),(k,l,n,m),(o,p,q,r)$ contains exactly two monochromatic edges. So vertices that are opposite of each other in these four-cycles  receive opposite colours. It is now easy to see that the only possible colourings are as shown in Fig.~\ref{fig80} where $\alpha,\beta \in \{\mbox{black},\mbox{white}\}$ and $\alpha'$ denotes the
opposite colour to $\alpha$ and $\beta'$ denotes the opposite colour to $\beta$. The result then follows.
\end{proof}

We are now able to prove that Problem~\ref{pro:planarmatching} is NP-complete.
\begin{theorem}\label{th:cubplanp}
Problem \ref{pro:planarmatching} is NP-complete.
\end{theorem}

\begin{proof}
Given an instance of Not-All-Equal 3SAT, construct the graph $G$ as in the proof of Proposition~\ref{prop:1}.
This graph can be drawn in the plane so that the only edges that cross are the identifying edges, each pair of identifying edges crosses at most once
and at most two edges cross at any point.

Now we replace the crossings one by one by replacing a pair of crossing edges by the uncrossing gadget. Suppose $\gamma$ and $\delta$ are two edges that cross. We delete $\gamma$ and $\delta$ and replace them with a copy of the uncrossing gadget attaching $w$ and $z_2$ to the end-vertices of $\gamma$, and $a$ and $z_4$ to the end-vertices of $\delta$. We will also call the four pendant edges $wv,ab,z_1z_2,z_3z_4$ in the uncrossing graph identifying edges. After each replacement we can draw the graph so that only identifying edges cross and such that there is one fewer crossing.
We continue until there are no more crossing edges. The final graph can be constructed in polynomial time and is planar and cubic.
Each original identifying edge in $G$ now corresponds to one or more identifying edges with each consecutive pair being on opposite sides of a copy of the uncrossing gadget. Lemma~\ref{le:uncrossgadget} shows that in a two-coloured perfect matching all of these edges must be monochromatic and all the end-vertices of these edges have the same colour.

Now the argument in Proposition~\ref{prop:1} shows that the final graph has a two-coloured perfect matching if and only if the instance of Not-All-Equal 3SAT is satisfiable.
\end{proof}

\section{$k$-$L(2,1)$-labelling for $k\ge4$ fixed}\label{sec:constgraph}
Let $G$ be a planar cubic graph. In order to establish our main result we will reduce Planar Cubic Two-Colourable Perfect Matching to Planar $k$-$L(2,1)$-Labelling for
planar graphs for each $k \ge 4$.
From any planar cubic graph $G$ forming an instance of Planar Cubic Two-Colourable Perfect Matching, we construct a graph $K$. As we see in the next section,
the basic form of $K$ does not depend on $k$ but is formed by constructing an auxiliary graph $H$ and then replacing each edge of $H$ by a gadget which does depend on $k$. In this section we define these gadgets and analyse certain $L(2,1)$-labellings of them. Each gadget has two distinguished vertices, which will always be labelled $u$ and $v$, corresponding to the end-vertices of the edge that is replaced in the auxiliary graph defined in the next section.
These two vertices have degree $k-1$ in $K$. Any vertex of degree $k-1$ must receive either label $0$ or $k$ in a $k$-$L(2,1)$-labelling
because these are the only possible labels for which there are $k-1$ labels remaining to label the neighbours of that vertex, so we will analyse the $k$-$L(2,1)$-labellings of these gadgets in which $u,v$ receive label $0$ or $k$.

\subsection{$\lambda_{2,1}(G) = 4$}
In this subsection the gadget $G_{4}$ is simply a path of length three.
More precisely the gadget $G_{4}$ is given by
$V(G_{4})=\{ u,a_u,a_v,v\}$ and $E(G_{4})=\{ua_u,a_ua_v,a_vv\}$.
This gadget is used in~\cite{fiala+kloks+kratochvil:fixed-parameter-lambda}, from where we get the following lemma.
\begin{lemma}\label{le:possib2}
There is a $4$-$L(2,1)$-labelling $L$ of $G_4$ with $L(u),L(v) \in \{0,4\}$ if and only if the following conditions are satisfied.
\begin{enumerate}
\item If $(L(u),L(v))=(0,0)$, then $(L(a_u),L(a_v)) \in \{(2,4),(4,2)\}$.
\item If $(L(u),L(v))=(4,4)$, then $(L(a_u),L(a_v)) \in \{(0,2),(2,0)\}$.
\item If $(L(u),L(v))=(4,0)$, then $(L(a_u),L(a_v)) =(1,3)$.
\item If $(L(u),L(v))=(0,4)$, then $(L(a_u),L(a_v)) =(3,1)$.
\end{enumerate}
\end{lemma}

\subsection{$\lambda_{2,1}(G) = 5$}

\begin{figure}[ht]
\psset{xunit=1cm,yunit=1cm,runit=1cm,labelsep=4pt}
\begin{center}
\begin{pspicture}(4,4)
\cnode*(0.5,3.5){0.075}{u}
\uput[90](0.5,3.5){$u$}
\cnode*(3.5,3.5){0.075}{v}
\uput[90](3.5,3.5){$v$}
\ncline{-}{u}{v}
\cnode*(1.5,3.5){0.075}{av}
\uput[90](1.5,3.5){$a_u$}
\cnode*(2.5,3.5){0.075}{au}
\uput[90](2.5,3.5){$a_v$}
\cnode*(2.5,3){0.075}{b1}
\uput[0](2.5,3){$b_v$}
\ncline{-}{b1}{au}
\cnode*(1.5,3){0.075}{b2}
\uput[180](1.5,3){$b_u$}
\ncline{-}{b2}{av}

\cnode*(2,2.5){0.075}{b3}
\uput[90](2,2.5){$c$}
\ncline{-}{b1}{b3}
\ncline{-}{b2}{b3}
\cnode*(2,2){0.075}{c3}
\uput[160](2,2){$d$}
\ncline{-}{c3}{b3}
\cnode*(2.5,2){0.075}{c4}
\uput[0](2.5,2){$e_3$}
\ncline{-}{c3}{c4}
\cnode*(2.5,1.5){0.075}{d6}
\uput[0](2.5,1.5){$e_1$}
\ncline{-}{c3}{d6}
\cnode*(1.5,1.5){0.075}{d4}
\uput[180](1.5,1.5){$e_2$}
\ncline{-}{c3}{d4}
\cnode*(2,1){0.075}{e2}
\uput[190](2,1){$f$}
\ncline{-}{e2}{d4}
\ncline{-}{e2}{d6}
\ncline{-}{c3}{d4}
\cnode*(1.5,0.5){0.075}{f3}
\uput[270](1.5,0.5){$g_1$}
\ncline{-}{f3}{e2}
\cnode*(2.5,0.5){0.075}{f4}
\uput[270](2.5,0.5){$g_2$}
\ncline{-}{f4}{e2}
\end{pspicture}
\end{center}
\caption{The edge gadget $G_{5}$.}\label{fig16}
\end{figure}
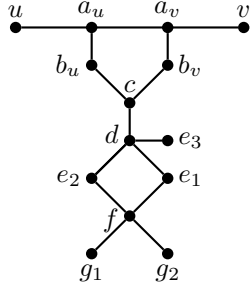
Let $G_{5}$ be the graph depicted in Fig.~\ref{fig16}.

\begin{lemma}\label{possib}
There is a $5$-$L(2,1)$-labelling $L$ of $G_5$ with $L(u),L(v) \in \{0,5\}$ if and only if the following conditions are satisfied.
\begin{enumerate}
\item If $(L(u),L(v))=(0,0)$, then $(L(a_u), L(a_v))\in\{(2,5), (5,2), (3,5), (5,3)\}$.
\item If $(L(u),L(v))=(5,5)$, then $(L(a_u), L(a_v))\in\{(3,0), (0,3), (2,0), (0,2)\}$.
\item If $(L(u),L(v))=(0,5)$, then  $(L(a_u), L(a_v))=(4,1)$.
\item If $(L(u),L(v))=(5,0)$, then  $(L(a_u), L(a_v))=(1,4)$.
\end{enumerate}
\end{lemma}
\begin{proof}\mbox{}
\begin{enumerate}
\item By the definition of $L(2,1)$-labelling, both $L(a_u)$ and $L(a_v)$ belong to the set $\{2,3,4,5\}$.
As $|L(a_u)-L(a_v)|\geq 2$, \[(L(a_u), L(a_v))\in\{(2,4), (4,2), (2,5), (5,2), (3,5), (5,3)\}.\]

Suppose for a contradiction that $(L(a_u), L(a_v))\in\{(2,4), (4,2)\}$. By symmetry, we may assume that
 $(L(a_u), L(a_v))=(2,4)$. Then $L(b_u)=5$ and $L(b_v)=1$.
The vertices $d$ and $f$ have degree four and thus must receive labels from $\{0,5\}$. Because $\dist(b_u,d)=2$ we must have $L(d)=0$ and because $\dist(d,f)=2$ we must have $L(f)=5$.
This implies that $\{L(e_1), L(e_2)\}=\{2,3\}$.
But then vertex $c$ cannot be labelled, giving a contradiction.
An $L(2,1)$-labelling is obtained if $(L(a_u),L(a_v))=(5,2)$ and $L(b_u)=3$, $L(b_v)=4$, $L(c)=0$, $L(d)=5$,  $L(e_3)=1$, $L(e_2)=2$, $L(e_1)=3$, $L(f)=0$
$L(g_1)=5$ and $L(g_2)=4$ or if $(L(a_u),L(a_v))=(5,3)$ and $L(b_u)=2$, $L(b_v)=1$,  $L(c)=4$, $L(d)=0$, $L(e_3)=5$, $L(e_2)=2$, $L(e_1)=3$, $L(f)=5$
$L(g_1)=0$ and $L(g_2)=1$. The other cases follow by symmetry.

\item Analogous to (i).

\item By the definition of $L(2,1)$-labelling, $L(a_u)\in \{2,3,4\}$ and $L(a_v)\in \{1,2,3\}$.
As $|L(a_u)-L(a_v)|\geq 2$, $(L(a_u), L(a_v))\in\{(4,1), (4,2), (3,1)\}$.
Suppose for a contradiction that  $(L(a_u), L(a_v))\neq (4,1)$.
By the label symmetry  $x\mapsto 5-x$, we may assume that  $(L(a_u), L(a_v))= (4,2)$.
Hence $L(b_u)=1$ and $L(b_v)=0$. Now the vertices $d$ and $f$ have degree four and thus
$L(d)=0,L(f)=5$ or $L(d)=5,L(f)=0$. As $\dist(b_v,d)=2$, $L(d)=5$ and as $\dist(d,f)=2$, $L(f)=0$.
Now $\{L(e_1), L(e_2)\}=\{2,3\}$. But then vertex $c$ cannot be labelled, a contradiction.
An $L(2,1)$-labelling is obtained if $(L(a_u),L(a_v))=(4,1)$ and $L(b_u)=2$, $L(b_v)=3$, $L(c)=0$, $L(d)=5$,  $L(e_3)=1$, $L(e_2)=2$, $L(e_1)=3$, $L(f)=0$
$L(g_1)=5$ and $L(g_2)=4$.

\item Analogous to (iii).
\end{enumerate}
\end{proof}

\subsection{$\lambda_{2,1}(G) \geq 6$}
In this subsection we introduce for any $k \ge 6$ the gadget $G_k$ and consider some of its $L(2,1)$-labellings.
The gadgets $G_6$ and $G_7$ are depicted in Fig.~\ref{fig21} and in Fig.~\ref{fig50}, respectively.
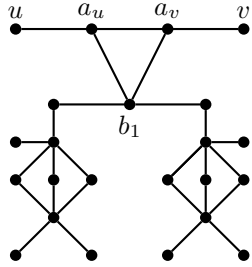
\begin{figure}[ht]
\psset{xunit=1cm,yunit=1cm,runit=1cm,labelsep=4pt}
\begin{center}
\begin{pspicture}(4,4)
\cnode*(0.5,3.5){0.075}{u}
\uput[90](0.5,3.5){$u$}
\cnode*(3.5,3.5){0.075}{v}
\uput[90](3.5,3.5){$v$}
\ncline{-}{u}{v}
\cnode*(1.5,3.5){0.075}{av}
\uput[90](1.5,3.5){$a_u$}
\cnode*(2.5,3.5){0.075}{au}
\uput[90](2.5,3.5){$a_v$}
\cnode*(2,2.5){0.075}{b1}
\uput[270](2,2.5){$b_1$}
\ncline{-}{b1}{av}
\ncline{-}{b1}{au}
\cnode*(1,2.5){0.075}{b2}
\ncline{-}{b1}{b2}
\cnode*(1,2){0.075}{c2}
\ncline{-}{b2}{c2}
\cnode*(0.5,2){0.075}{c1}
\ncline{-}{c1}{c2}
\cnode*(0.5,1.5){0.075}{d1}
\ncline{-}{d1}{c2}
\cnode*(1,1.5){0.075}{d2}
\ncline{-}{d2}{c2}
\cnode*(1.5,1.5){0.075}{d3}
\ncline{-}{d3}{c2}
\cnode*(1,1){0.075}{e1}
\ncline{-}{d1}{e1}
\ncline{-}{d2}{e1}
\ncline{-}{d3}{e1}
\cnode*(0.5,0.5){0.075}{f1}
\ncline{-}{f1}{e1}
\cnode*(1.5,0.5){0.075}{f1}
\ncline{-}{f1}{e1}

\cnode*(3,2.5){0.075}{b3}
\ncline{-}{b1}{b3}
\cnode*(3,2){0.075}{c3}
\ncline{-}{c3}{b3}
\cnode*(3.5,2){0.075}{c4}
\ncline{-}{c3}{c4}
\cnode*(3.5,1.5){0.075}{d6}
\ncline{-}{c3}{d6}
\cnode*(3,1.5){0.075}{d5}
\ncline{-}{c3}{d5}
\cnode*(2.5,1.5){0.075}{d4}
\ncline{-}{c3}{d4}
\cnode*(3,1){0.075}{e2}
\ncline{-}{e2}{d5}
\ncline{-}{e2}{d4}
\ncline{-}{e2}{d6}
\ncline{-}{c3}{d4}
\cnode*(2.5,0.5){0.075}{f3}
\ncline{-}{f3}{e2}
\cnode*(3.5,0.5){0.075}{f4}
\ncline{-}{f4}{e2}
\end{pspicture}
\end{center}
\caption{The edge gadget $G_{6}$.}\label{fig21}
\end{figure}

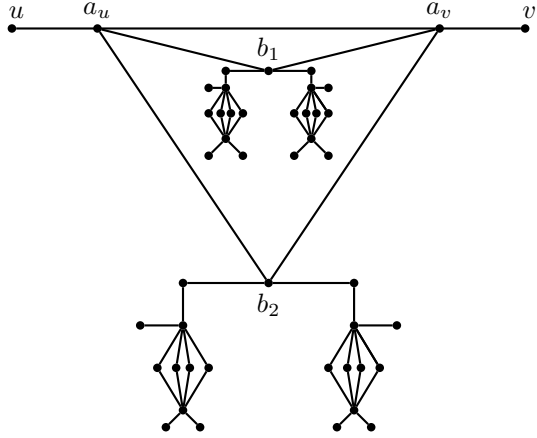
\begin{figure}[ht]
\psset{xunit=2.25cm,yunit=2.25cm,runit=2.25cm,labelsep=4pt}
\begin{center}
\begin{pspicture}(4,3)
\cnode*(0.5,2.5){0.025}{u}
\uput[90](0.5,2.5){ $u$}
\cnode*(3.5,2.5){0.025}{v}
\uput[90](3.5,2.5){ $v$}
\ncline{-}{u}{v}
\cnode*(1,2.5){0.025}{av}
\uput[90](1,2.5){$a_u$}
\cnode*(3,2.5){0.025}{au}
\uput[90](3,2.5){$a_v$}
\cnode*(2,1){0.025}{b1}
\uput[270](2,1){$b_2$}
\ncline{-}{b1}{av}
\ncline{-}{b1}{au}
\cnode*(1.5,1){0.025}{b2}
\ncline{-}{b1}{b2}
\cnode*(1.5,0.75){0.025}{c2}
\ncline{-}{b2}{c2}
\cnode*(1.25,0.75){0.025}{c1}
\ncline{-}{c1}{c2}
\cnode*(1.35,0.5){0.025}{d1}
\ncline{-}{d1}{c2}
\cnode*(1.46,0.5){0.025}{d2}
\cnode*(1.54,0.5){0.025}{d2x}
\ncline{-}{d2}{c2}
\ncline{-}{d2x}{c2}
\cnode*(1.65,0.5){0.025}{d3}
\ncline{-}{d3}{c2}
\cnode*(1.5,0.25){0.025}{e1}
\ncline{-}{d1}{e1}
\ncline{-}{d2}{e1}
\ncline{-}{d2x}{e1}
\ncline{-}{d3}{e1}
\cnode*(1.4,0.15){0.025}{f1}
\ncline{-}{f1}{e1}
\cnode*(1.6,0.15){0.025}{f1}
\ncline{-}{f1}{e1}

\cnode*(2.5,1){0.025}{b3}
\ncline{-}{b1}{b3}
\cnode*(2.5,0.75){0.025}{c3}
\ncline{-}{c3}{b3}
\cnode*(2.75,0.75){0.025}{c4}
\ncline{-}{c3}{c4}
\cnode*(2.35,0.5){0.025}{d6}
\ncline{-}{c3}{d6}
\cnode*(2.46,0.5){0.025}{d5}
\cnode*(2.54,0.5){0.025}{d5x}
\ncline{-}{c3}{d5}
\ncline{-}{c3}{d5x}
\cnode*(2.65,0.5){0.025}{d4}
\ncline{-}{c3}{d4}
\cnode*(2.5,0.25){0.025}{e2}
\ncline{-}{e2}{d5}
\ncline{-}{e2}{d5x}
\ncline{-}{e2}{d4}
\ncline{-}{e2}{d6}
\ncline{-}{c3}{d4}
\cnode*(2.4,0.15){0.025}{f3}
\ncline{-}{f3}{e2}
\cnode*(2.6,0.15){0.025}{f4}
\ncline{-}{f4}{e2}
\cnode*(2,2.25){0.025}{b11}
\uput[90](2,2.25){$b_1$}
\ncline{-}{b11}{av}
\ncline{-}{b11}{au}
\cnode*(1.75,2.25){0.025}{b21}
\ncline{-}{b11}{b21}
\cnode*(1.75,2.15){0.025}{c21}
\ncline{-}{b21}{c21}
\cnode*(1.65,2.15){0.025}{c11}
\ncline{-}{c11}{c21}
\cnode*(1.65,2){0.025}{d11}
\ncline{-}{d11}{c21}
\cnode*(1.72,2){0.025}{d21}
\cnode*(1.78,2){0.025}{d212}
\ncline{-}{d21}{c21}
\ncline{-}{d212}{c21}
\cnode*(1.85,2){0.025}{d31}
\ncline{-}{d31}{c21}
\cnode*(1.75,1.85){0.025}{e11}
\ncline{-}{d11}{e11}
\ncline{-}{d21}{e11}
\ncline{-}{d31}{e11}
\ncline{-}{d212}{e11}
\cnode*(1.65,1.75){0.025}{f11}
\ncline{-}{f11}{e11}
\cnode*(1.85,1.75){0.025}{f11}
\ncline{-}{f11}{e11}

\cnode*(2.25,2.25){0.025}{b31}
\ncline{-}{b11}{b31}
\cnode*(2.25,2.15){0.025}{c31}
\ncline{-}{c31}{b31}
\cnode*(2.35,2.15){0.025}{c41}
\ncline{-}{c31}{c41}
\cnode*(2.15,2){0.025}{d61}
\ncline{-}{c31}{d61}
\cnode*(2.22,2){0.025}{d51}
\cnode*(2.28,2){0.025}{d512}
\ncline{-}{c31}{d51}
\ncline{-}{c31}{d512}
\cnode*(2.35,2){0.025}{d41}
\ncline{-}{c31}{d41}
\cnode*(2.25,1.85){0.025}{e21}
\ncline{-}{e21}{d51}
\ncline{-}{e21}{d41}
\ncline{-}{e21}{d61}
\ncline{-}{c31}{d41}
\ncline{-}{e21}{d512}
\cnode*(2.15,1.75){0.025}{f31}
\ncline{-}{f31}{e21}
\cnode*(2.35,1.75){0.025}{f41}
\ncline{-}{f41}{e21}
\end{pspicture}
\end{center}
\caption{The edge gadget $G_7$.}\label{fig50}
\end{figure}

Let $H'$ be defined as follows, see Fig.~\ref{fig22}.
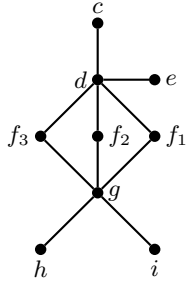
\begin{figure}
\psset{xunit=1cm,yunit=1cm,runit=1cm,labelsep=4pt}
\begin{center}
\begin{pspicture}(0,1)(4,4.75)
\cnode*(2,4.5){0.075}{c}
\uput[90](2,4.5){\small $c$}
\cnode*(2,3.75){0.075}{d}
\uput[180](2,3.75){\small $d$}
\ncline{-}{c}{d}
\cnode*(2.75,3.75){0.075}{e}
\uput[0](2.75,3.75){\small $e$}
\ncline{-}{d}{e}
\cnode*(2.75,3){0.075}{f1}
\uput[0](2.75,3){\small $f_{1}$}
\ncline{-}{f1}{d}
\cnode*(2,3){0.075}{f2}
\uput[0](2,3){\small $f_{2}$}
\ncline{-}{f2}{d}
\cnode*(1.25,3){0.075}{f3}
\uput[180](1.25,3){\small $f_{3}$}
\ncline{-}{f3}{d}
\cnode*(2,2.25){0.075}{g}
\uput[0](2,2.25){\small $g$}
\ncline{-}{g}{f1}
\ncline{-}{g}{f2}
\ncline{-}{f3}{g}
\cnode*(1.25,1.5){0.075}{h}
\uput[270](1.25,1.5){\small $h$}
\ncline{-}{h}{g}
\cnode*(2.75,1.5){0.075}{i}
\uput[270](2.75,1.5){\small $i$}
\ncline{-}{i}{g}
\end{pspicture}
\end{center}
\caption{The graph $H'$ for $k=6$.}\label{fig22}
\end{figure}

\begin{align*}
V(H')=\{ c,d,e,f_1,...,f_{k-3},g,h,i \}
\end{align*}
and
\begin{align*}
E(H')= \{ &cd,de,hg,gi,gf_1,...,gf_{k-3},df_{1},...,df_{k-3}  \}.
\end{align*}

\begin{lemma}\label{le:labelhn}
For any $k \ge 6$ the graph $H'$ is planar and there exists an $L(2,1)$-labelling $L$ of $H'$ with span $k$ if and only if
\begin{align*}
(L(c),L(d)) \in \{(0,k),(1,k),(k-1,0),(k,0) \}.
\end{align*}
\end{lemma}

\begin{proof}
As seen from Fig.~\ref{fig22}, $H'$ is planar for $k=6$. For any higher $k$ we need to connect $k-6$ paths of length two at $d$ and $g$
to the graph $H'$ in Fig.~\ref{fig22}. So $H'$ is planar for any $k \ge 6$.

If $L$ is a $k$-$L(2,1)$-labelling of $H'$ then $\{L(g),L(d)\} = \{0,k\}$ because $g$ and $d$ have degree $k-1$ and $\dist(g,d)=2$.
It follows that $\{f_{1},...,f_{k-3}\} = \{ 2,...,k-2 \}$.
If $L(d)=0$ and $L(g)=k$, then $\{ L(h),L(i) \}= \{ 1,0 \}$ and $\{L(c),L(e)\}=\{k-1,k\}$. Similarly if $L(d)=k$ and $L(g)=0$, then $\{ L(h),L(i) \}= \{ k-1,k \}$ and $\{L(c),L(e)\}=\{0,1\}$.
\end{proof}

We now define the graph $G_k$ for $k \ge 6$. Take a path of length three with vertices $u,u_a,a_v,v$ and edges $ua_u, a_ua_v,a_vv$. Add vertices $b_1,...,b_{k-5}$, with each joined to $a_u$ and $a_v$. Now for each $i=1,...,k-5$, add two copies of $H'$ with the vertex labelled $c$ in each copy joined to $b_i$. So each of $b_1,...,b_{k-5}$ has degree four.
To refer to the vertices in the two copies of $H'$ attached to $b_i$, we add a subscript of $(l,i)$ to the name of the vertices in one copy of $H'$ and $(r,i)$ to the name of the vertices in the other copy of $H'$. Notice that $G_k$ is planar.

\begin{lemma}\label{th:col1}
There exists a $k$-$L(2,1)$-labelling $L$ of $G_k$ with $L(u)=L(v)=0$ if and only if $(L(a_v),L(a_u)) \in \{ (2,k),(k,2),(k-2,k),(k,k-2) \}$.
\end{lemma}

\begin{proof}

First notice that we need $k-5$ different colours to colour the vertices $b_1,...,b_{k-5}$ as they are all at distance two from each other.
We first show that there exists a $k$-$L(2,1)$-labelling $L$ of $G_k$ with $L(u)=L(v)=0$ and $(L(a_u),L(a_v)) \in \{ (2,k),(k,2),(k-2,k),(k,k-2)  \}$.
The first case is when $L(a_u)=2$ and $L(a_v)=k$. Take $L(b_j)=j+3$ for $j=1,...,k-5$. A $k$-$L(2,1)$-labelling is obtained by setting
$L(d_{l,j})=L(d_{r,j})=k$ and $L(c_{l,j})=0$ and $L(c_{r,j})=1$ for $1 \le j \le k-5$ and then using Lemma~\ref{le:labelhn} to give a valid labelling  of the rest of the graph.

A similar argument shows that we may take $(L(a_u),L(a_v))=(k,2)$.

The second case is $L(a_u)=k-2$ and $L(a_v)=k$. Take $L(b_j)=j+1$ for $j=1,...,k-5$. A $k$-$L(2,1)$-labelling is obtained by setting $L(d_{l,j})=0$, $L(d_{r,j})=k$ and $L(c_{l,j})=k-1$ and $L(c_{r,j})=0$ for $1 \le j \le k-5$ and then using Lemma~\ref{le:labelhn} to give a valid labelling of the rest of the graph.

A similar argument shows that we may take $(L(a_u),L(a_v))=(k,k-2)$.

Next we show that there is no $k$-$L(2,1)$-labelling $L$ of $G_k$ with $L(u)=L(v)=0$ and $(L(a_v),L(a_u)) \not\in \{ (2,k),(k,2),(k-2,k),(k,k-2) \}$.
Assume without loss of generality that $L(a_u) < L(a_v)$.
Suppose first that $3 \le L(a_u) \le k-3$ and $L(a_v)=k$.
Note that for any $j$ by considering the proximity of $b_j$ to $u$ and $a_v$, we see that $b_j$ cannot be labelled with $0,k-1$ or $k$.
Furthermore we cannot have $b_j=1$ because then  $\{L(c_{l,j}),L(c_{r,j})\}=\{k,k-1\}$. But they are both at distance two from $a_v$ which is labelled $k$, so this is not possible.
So $b_1,...,b_{k-5}$ must receive distinct labels from $\{2,...,k-2\} \backslash \{ L(a_u)-1, L(a_u), L(a_u)+1 \}$, but this only gives $k-6$ labels
which is not enough.

Now suppose $L(a_v) \not= k$ and $2 \le L(a_u) \le L(a_v)-2 \le k-3$.
Note that for any $j$, $b_i$ cannot be labelled with $0$.
Furthermore $b_j$ cannot be labelled $k$ as
then $\{L(c_{l,j}),L(c_{r,j})\}=\{0,1\}$ and $L(d_{l,j})=L(d_{r,j})=k$ but this is invalid.
So  $b_1,...,b_{k-5}$ must receive distinct labels from  $\{ 1,...,k-1\} \backslash \{ L(a_u)-1,L(a_u),L(a_u)+1,L(a_v)-1,L(a_v),L(a_v)+1\}$.
This is only possible if $L(a_u)=k-3$ and $L(a_v)=k-1$. Then $\{L(b_1),...,L(b_{k-5})\} = \{ 1,...,k-5\}$. However if $L(b_j)=1$ then $\{L(c_{l,j}),L(c_{r,j})\}=\{k-1,k\}$. But this is invalid as $L(a_v)=k-1$.
\end{proof}

Analogously we obtain the following lemma.
\begin{lemma}\label{th:col2}
There exists a $k$-$L(2,1)$-labelling $L$ of $G_k$ with $L(u)=L(v)=k$ if and only if  $(L(a_v),L(a_u)) \in \{ (2,0),(0,2),(k-2,0),(0,k-2) \}$.
\end{lemma}

\begin{lemma}\label{th:col3}
There exists a $k$-$L(2,1)$-labelling $L$ of $G_k$ with $L(u)=k$ and $L(v)=0$ if and only if  $(L(a_v),L(a_u)) =(k-1,1)$.
\end{lemma}

\begin{proof}
Let  $l_1=\min\{L(a_u),L(a_v)\}$ and $l_2=\max\{ L(a_u),L(a_v)\}$.
The vertices  $b_1,...,b_{k-5}$ must be labelled with distinct labels from $S= \{ 1,...,k-1 \}\backslash \{ l_1-1,l_1,l_1+1,l_2-1,l_2,l_2+1\}$.
The only way that this set can contain $k-5$ elements is when $(l_1,l_2)=(1,k-1)$, $(l_1,l_2)=(1,3)$ or $(l_1,l_2)=(k-3,k-1)$.

We first show that there exists a $k$-$L(2,1)$-labelling $L$ of $G_k$ with $L(u)=k,L(v)=0$ and $(L(a_v),L(a_u)) =(k-1,1)$.

We need $\{L(b_1),...,L(b_{k-5})\}=\{3,...,k-3\}$.
Then let $L(c_{l,j})=0$, $L(d_{l,j})=k$, $L(c_{r,j})=k$ and $L(d_{l,j})=0$ for $1\le j \le k-5$.
So by Lemma~\ref{le:labelhn} we obtain a valid labelling.

Next we show that there is no $k$-$L(2,1)$-labelling of $G_k$ with $L(u)=k$ and $L(v)=0$ and $(L(a_v),L(a_u))\not=(k-1,1)$.

Assume that $l_1=1$ and $l_2=3$ so $L(a_v)=3$ and $L(a_u)=1$. Then the vertices $b_1,...,b_{k-5}$ must take distinct labels from
$ \{  5,...,k-1\}$. However if $b_j$ is labelled $k-1$ the only label $c_{l,j}$ and $c_{r,j}$
can be labelled with is $0$ but $c_{l,j}$ and $c_{r,j}$ must have distinct labels. Therefore this labelling is not possible.

Now suppose $l_1=k-3$ and $l_2=k-1$ then $L(a_v)=k-1$ and $L(a_u)=k-3$. Then the vertices $b_1,...,b_{k-5}$ must take distinct labels
from $ \{  1,...,k-5\}$. However if $b_j$ is labelled $1$ the only label $c_{l,j}$ and $c_{r,j}$
can be labelled with is $k$ but $c_{l,j}$ and $c_{r,j}$ must have distinct labels. Therefore this labelling is not possible.
\end{proof}

\subsection{Summary}

The following theorem summarises the results of this section and follows immediately from
Lemmas~\ref{le:possib2},~\ref{possib},~\ref{th:col1},~\ref{th:col2} and~\ref{th:col3}.

\begin{theorem}\label{th:colgadget}
Let $k\ge 4$ be fixed. There is a $k$-$L(2,1)$-labelling $L$ of $G_k$ with $L(u),L(v) \in \{0,k\}$ if and only if the following conditions are satisfied.
\begin{enumerate}
\item If $(L(u),L(v))=(0,0)$ then
\[(L(a_u),L(a_v))\in \{ (2,k),(k,2),(k-2,k),(k,k-2) \}.\]
\item If $(L(u),L(v))=(k,k)$ then \[(L(a_u),L(a_v))\in \{ (2,0),(0,2),(k-2,0),(0,k-2) \}.\]
\item If $(L(u),L(v))=(k,0)$ then $(L(a_u),L(a_v))=(1,k-1)$.
\item If $(L(u),L(v))=(0,k)$ then $(L(a_u),L(a_v))=(k-1,1)$.
\end{enumerate}
\end{theorem}

\section{Planar $k$-$L(2,1)$-labelling is NP-complete for $k \ge 4$}\label{sec:npcom4}
We reduce Planar Cubic Two-Colourable Perfect Matching to Planar $k$-$L(2,1)$-Labelling. Suppose we are given a cubic planar graph $G$ corresponding to an instance of Planar Cubic Two-Colourable Perfect Matching. From $G$ we construct a graph $K$ which has the property that $K$ has a $k$-$L(2,1)$-labelling if and only if $G$ has a two-coloured perfect matching.

In order to show this we also construct an auxiliary graph $H$ and define what we call a \emph{coloured orientation}. Then we show that $G$ has a two-coloured
perfect matching if and only if $H$ has a coloured orientation and finally that $H$ has a coloured orientation if and only if $K$ has a
$k$-$L(2,1)$-labelling.

$H$ is obtained by replacing every edge of $G$ with the gadget as depicted in Fig.~\ref{fig610}, where the end-vertices of the edge being replaced are $u,v$.
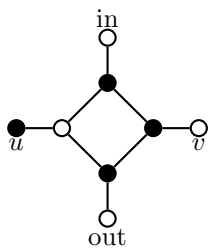
\begin{figure}[ht]
\psset{xunit=0.2cm,yunit=0.2cm,runit=0.2cm,labelsep=4pt}
\psset{arrowscale=1.5}
\begin{center}
\begin{pspicture}(13,15)
\cnode*(0.5,7.5){0.6}{a}
\uput[270](0.5,7.5){$u$}
\cnode(3.5,7.5){0.6}{d}
\ncline{-}{d}{a}
\cnode*(6.5,10.5){0.6}{i}
\cnode*(6.5,4.5){0.6}{j}
\cnode*(9.5,7.5){0.6}{o}
\ncline{-}{d}{j}
\ncline{-}{i}{d}
\ncline{-}{o}{i}
\ncline{-}{j}{o}
\cnode(12.5,7.5){0.6}{r}
\uput[270](12.5,7.5){$v$}
\ncline{-}{o}{r}
\cnode(6.5,13.5){0.6}{u}
\uput[90](6.5,13.5){in}
\ncline{-}{i}{u}
\cnode(6.5,1.5){0.6}{z}
\uput[270](6.5,1.5){out}
\ncline{-}{j}{z}
\end{pspicture}
\end{center}
\caption{An auxiliary edge.}\label{fig610}
\end{figure}

 The gadget has two special vertices labelled \emph{in} and \emph{out}, which we call the \emph{invertex} and the \emph{outvertex} and we explain in a moment. The edges incident with them are called the \emph{inedge} and \emph{outedge}, respectively. We use the phrase \emph{auxiliary edge} to refer to a subgraph of $H$ that has replaced an edge of $G$, that is, any of the copies of the gadget from Fig.~\ref{fig610}. A coloured orientation of an auxiliary graph $H$ is a colouring of the vertices of $H$ with black and white and an orientation of some of the edges satisfying
certain properties. The indegree and outdegree of a vertex $v$ are the number of edges oriented towards $v$ and the number of edges oriented away from $v$, respectively. Unoriented edges are not counted towards indegree and outdegree. A coloured orientation must satisfy the following properties.

\begin{itemize}
\item Every vertex is adjacent to at most one vertex of the opposite colour.
\item An edge is oriented if and only if it is monochromatic.
\item Every vertex except those labelled \emph{out} has outdegree at most two and indegree at most one.
\item Every vertex labelled \emph{out} has indegree zero.
\end{itemize}
We say a coloured orientation is \emph{good} if every vertex of degree three is adjacent to precisely one vertex of the opposite colour and has indegree and
outdegree one.

\begin{lemma}\label{le:main1}
Let $G$ be a cubic planar graph and let $H$ be the corresponding auxiliary graph. If $G$ has a two-coloured perfect matching then $H$ has a good coloured orientation.
\end{lemma}
\begin{proof}
First colour the vertices of $H$ that were present in $G$ with the same colour that they receive in $G$.

We next colour the vertices of auxiliary edges where both end-vertices of the corresponding edge in $G$ receive the same colour.
The in- and outvertex receive the same colour as the end-vertices of the corresponding edge in $G$ and the vertices on the four-cycle receive the opposite colour. We orient this cycle to form a directed cycle, see Fig.~\ref{fig61}.
\begin{figure}[ht]
\psset{xunit=0.2cm,yunit=0.2cm,runit=0.2cm,labelsep=4pt}
\psset{arrowscale=1.5}
\begin{center}
\begin{pspicture}(13,15)
\cnode*(0.5,7.5){0.6}{a}
\uput[270](0.5,7.5){$u$}
\cnode(3.5,7.5){0.6}{d}
\ncline{-}{d}{a}
\cnode(6.5,10.5){0.6}{i}
\cnode(6.5,4.5){0.6}{j}
\cnode(9.5,7.5){0.6}{o}
\ncline{->}{d}{j}
\ncline{->}{i}{d}
\ncline{->}{o}{i}
\ncline{->}{j}{o}
\cnode*(12.5,7.5){0.6}{r}
\uput[270](12.5,7.5){$v$}
\ncline{-}{o}{r}
\cnode*(6.5,13.5){0.6}{u}
\uput[90](6.5,13.5){in}
\ncline{-}{i}{u}
\cnode*(6.5,1.5){0.6}{z}
\uput[270](6.5,1.5){out}
\ncline{-}{j}{z}
\end{pspicture}
\end{center}
\caption{Good coloured orientation of an auxiliary edge if $u$ and $v$ receive the same colour.} \label{fig61}
\end{figure}
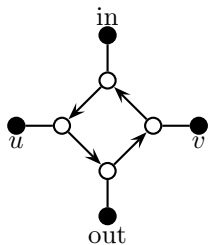

Now we colour all the other vertices and orient edges as follows. Vertices remaining uncoloured all belong to auxiliary edges for which the end-vertices of the corresponding edge in $G$ are coloured differently in $G$.
We start by choosing an auxiliary edge $e$ between a black vertex $v$ and a white vertex $w_1$ of $G$.
In $G$, $v$ has two white neighbours and one black neighbour. Call the other white neighbour $w_2$. Colour the outvertex of $e$ black and
orient the edge incident with it away from the outvertex. Now follow the shortest path from the outvertex, through $v$ and to the invertex of the
auxiliary edge $vw_2$. Colour every uncoloured vertex on this path black and orient every edge consistently with the path.
At each stage of this colouring/orientation process we will colour and orient a path like this from an outvertex of an auxiliary edge, through a vertex present in $G$ to an invertex of a neighbouring auxiliary edge, see Fig.~\ref{fig70}.
It only remains to describe how to choose the outvertex and invertex pair forming the end-vertices of each path.
The first pair is chosen as above. Otherwise, if at some stage, we colour an invertex of an auxiliary edge $f$ with colour $c$ and the outvertex $f$ is still uncoloured then at the next stage we form a path from the outvertex of $f$ colouring the vertices on it with the opposite colour to $c$. If the outvertex of $f$ has already been coloured then we choose another auxiliary edge for which the outvertex is uncoloured. This method ensures that at each stage there is at most one auxiliary edge with the outvertex coloured and the invertex uncoloured and at most one auxiliary edge with the outvertex not coloured but the invertex coloured. Such an uncoloured outvertex is always the next one to be
coloured.

\begin{figure}
\psset{xunit=1.5cm,yunit=1.5cm,runit=1.5cm}
\psset{arrowscale=1.5}
\begin{center}
\begin{pspicture}(4,4)
\cnode*(2,2){0.075}{a}
\cnode(2,2.5){0.075}{b}
\rput(2.15,2.1){\small $v$}
\cnode(2,3){0.075}{c}
\cnode*(2,3.5){0.075}{d}
\rput(2,3.655){\small $w_3$}
\cnode(1.75,2.75){0.075}{e}
\cnode*(1.25,2.75){0.075}{g}
\cnode(2.25,2.75){0.075}{f}
\cnode*(2.75,2.75){0.075}{j}
\ncline{-}{a}{b}
\ncline{->}{b}{e}
\ncline{->}{f}{b}
\ncline{->}{e}{c}
\ncline{->}{c}{f}
\ncline{-}{c}{d}
\ncline{-}{e}{g}
\ncline{-}{f}{j}
\cnode*(2.53,1.47){0.075}{bb}
\ncline{->}{bb}{a}
\cnode*(2.88,1.47){0.075}{ee}
\cnode*(3.18,1.77){0.075}{gg}
\ncline{->}{gg}{ee}
\ncline{->}{ee}{bb}
\cnode(2.88,1.12){0.075}{cc}
\cnode(3.23,0.77){0.075}{dd}
\rput(3.4,0.6){\small $w_1$}
\ncline{->}{dd}{cc}
\ncline{-}{ee}{cc}
\cnode(2.53,1.12){0.075}{ff}
\cnode(2.23,0.82){0.075}{hh}
\ncline{->}{ff}{hh}
\ncline{-}{ff}{bb}
\ncline{->}{cc}{ff}
\cnode*(1.47,1.47){0.075}{bbb}
\ncline{->}{a}{bbb}
\cnode(1.47,1.12){0.075}{eee}
\cnode(1.77,0.82){0.075}{ggg}
\ncline{->}{ggg}{eee}
\ncline{-}{eee}{bbb}
\cnode*(1.12,1.47){0.075}{fff}
\cnode(1.12,1.12){0.075}{ccc}
\cnode(0.77,0.77){0.075}{ddd}
\rput(0.65,0.6){\small $w_2$}
\ncline{->}{ccc}{ddd}
\ncline{->}{eee}{ccc}
\ncline{-}{ccc}{fff}
\ncline{->}{bbb}{fff}
\cnode*(0.82,1.77){0.075}{hhh}
\ncline{->}{fff}{hhh}
\end{pspicture}
\end{center}
\caption{Assignment of a good coloured orientation on $H$.}\label{fig70}
\end{figure}
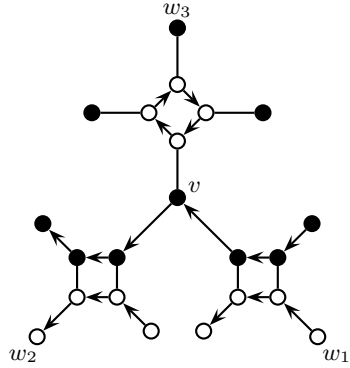

Due to the construction process, every vertex of $H$ which is also present in $G$ is adjacent to exactly one vertex of the opposite colour and has indegree and outdegree one. Clearly the same is true for all vertices of degree three of auxiliary edges where both end-vertices of the corresponding edge in $G$ receive the same colour. Now consider an auxiliary edge which corresponds to a dichromatic edge $e$ in $G$. Due to the colouring/orientation process the invertex and outvertex must receive opposite colours and the shortest path from each of them to the end-vertices of $e$ with the same colour is monochromatic.
It follows that all vertices of degree three on the auxiliary edge must be adjacent to exactly one vertex of the opposite colour and have indegree and outdegree one, see Fig.~\ref{fig70}.
Therefore the method yields a good coloured orientation.
\end{proof}

\begin{lemma}\label{le:goodorient}
Let $G$ be a cubic planar graph and let $H$ be the corresponding auxiliary graph. If $H$ has a coloured orientation, then it has a good coloured
orientation.
\end{lemma}
\begin{proof}
Consider the possible coloured orientations of an auxiliary edge. Because each vertex is adjacent to at most one of the opposite colour the only ways in which the four-cycle of an auxiliary edge may be coloured are with all four vertices receiving the same colour or with a pair of adjacent vertices receiving one colour and the other pair receiving the opposite colour.
In the first case we may change the colours of the invertex and outvertex (if necessary) to be the opposite colour to that of the vertices in the four-cycle. (We also remove the orientation of the inedge and outedge if necessary.)
In this way both the inedge and the outedge are dichromatic. In the second case the fact that each vertex is adjacent to at most one vertex of the
opposite colour forces both the invertex and the outvertex to have the same colour as their neighbour.

From now on we will assume we have a coloured orientation with each auxiliary edge being coloured in this way.
We will show that if $H$ has a coloured orientation then it has a good coloured orientation.
Let $H'$ be formed from $H$ by deleting all the dichromatic, or equivalently unoriented edges, and consider a connected component $C$ of $H'$. In $H'$ every vertex has outdegree at most two and indegree at most one.
So $C$ is either an isolated vertex, a directed circuit with a number of trees rooted on the circuit and directed away from the circuit or a directed rooted tree in which all edges are oriented away from the root.
Bearing in mind the constraints on the in- and outdegree of the vertices, we see that every vertex of degree three in the auxiliary graph has total degree at least two in $H'$. Leaves of $H'$ correspond to invertices and roots with degree one correspond to either invertices or outvertices.
Notice that the isolated vertices of $H'$ can only be invertices or outvertices and by the remarks at the beginning of the proof exactly half of the isolated vertices are invertices. Hence the numbers of invertices and outvertices appearing in $H'$ that are not isolated are equal. So the number of leaves of $H'$ is at most the number of roots of tree components. Consequently each connected component of $H'$, that is not just an isolated vertex, is either a path beginning at an ouvertex and ending at an invertex or a directed circuit. So every vertex of degree three in the auxiliary graph has one out-neighbour, one in-neighbour and one incident unoriented edge. Therefore the coloured orientation is good.
\end{proof}

\begin{lemma}\label{le:main2}
Let $G$ be a cubic planar graph and let $H$ be the corresponding auxiliary graph. If $H$ has a good coloured orientation then $G$ has a two-coloured perfect matching.
\end{lemma}

\begin{proof}
Consider a vertex $v$ of $G$ and let $w_1,w_2$ and $w_3$ be its neighbours in $G$. Suppose without loss of generality that $v$ is coloured
black in the good coloured orientation of $H$. We will show that in $H$, two of the vertices $w_1,w_2,w_3$ are coloured white and one is coloured black. Then we only need to assign to any vertex in $G$ the colour it receives in the good coloured orientation in $H$ to obtain a two-coloured perfect matching of $G$.

Vertex $v$ has two black neighbours and one white neighbour in $H$.
In the proof of Lemma~\ref{le:goodorient} we showed that in a good coloured orientation a terminal vertex of an auxiliary edge receives the opposite colour to the unique neighbour in the auxiliary edge of the other terminal vertex of the auxiliary edge. Thus two of the vertices $w_1,w_2,w_3$ must be coloured white and the other one black.
\end{proof}

Now given an instance $G$ of Planar Cubic Two-Colourable Perfect Matching, we define an instance $K$ of $k$-$L(2,1)$-labelling. First form the auxiliary graph $H$. For every vertex $v$ of $H$ add sufficient vertices of degree one with edges joining them to $v$ to ensure that $v$ has degree $k-1$.
Now replace each edge that was originally present in $H$ by the gadget $G_k$ identifying the vertices $u,v$ of $G_k$ with the two end-vertices of edges of $H$ being replaced. Finally for each outvertex $v$ choose a neighbour $w$ of $v$ with degree one and add $k-2$ vertices of degree one joined to $w$.
To illustrate this, suppose that $k=4$ and $v$ is adjacent to $w_1,w_2,w_3$ in $G$. In Fig.~\ref{fig300} we show how the neighbourhood of $v$ is modified in $K$.
Note that $K$ can be constructed from $G$ in time $O(n)$.

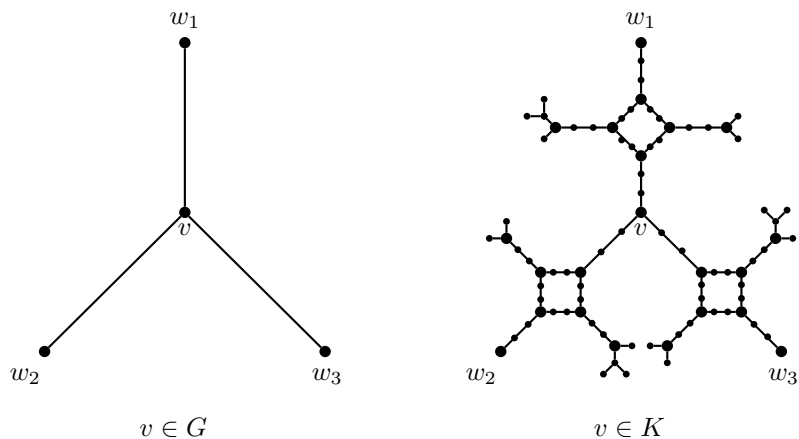
\begin{figure}[ht]
\begin{center}
\psset{xunit=1.5cm,yunit=1.5cm,runit=1.5cm}
\begin{pspicture}(1.1,0)(9,4)
\cnode*(7,2){0.05}{a}
\rput(7,1.85){$v$}
\cnode*(7,2.17){0.03}{r}
\cnode*(7,2.34){0.03}{r}
\cnode*(7,2.5){0.05}{b}
\cnode*(6.92,2.58){0.03}{r}
\cnode*(6.83,2.64){0.03}{r}
\cnode*(7.08,2.58){0.03}{r}
\cnode*(7.16,2.64){0.03}{r}
\cnode*(7,3){0.05}{c}
\cnode*(7.08,2.91){0.03}{r}
\cnode*(7.16,2.83){0.03}{r}
\cnode*(7,3.17){0.03}{r}
\cnode*(7,3.34){0.03}{r}
\cnode*(7,3.5){0.05}{d}
\rput(7.0,3.7){$w_1$}
\rput(3.0,3.7){$w_1$}
\cnode*(6.75,2.75){0.05}{e}
\cnode*(6.83,2.83){0.03}{r}
\cnode*(6.91,2.91){0.03}{r}
\cnode*(6.58,2.75){0.03}{r}
\cnode*(6.41,2.75){0.03}{r}
\cnode*(6.25,2.75){0.05}{g}
\cnode*(6.15,2.65){0.03}{h}
\cnode*(6.15,2.85){0.03}{i}
\cnode*(6.15,2.85){0.03}{i}
\cnode*(6.15,3){0.03}{i1}
\cnode*(6.0,2.85){0.03}{i2}
\ncline{-}{i}{i1}
\ncline{-}{i}{i2}
\cnode*(7.25,2.75){0.05}{f}
\cnode*(7.42,2.75){0.03}{r}
\cnode*(7.59,2.75){0.03}{r}
\cnode*(7.75,2.75){0.05}{j}
\cnode*(7.85,2.85){0.03}{k}
\cnode*(7.85,2.65){0.03}{l}
\ncline{-}{a}{b}
\ncline{-}{b}{e}
\ncline{-}{b}{f}
\ncline{-}{e}{c}
\ncline{-}{f}{c}
\ncline{-}{c}{d}
\ncline{-}{e}{g}
\ncline{-}{g}{h}
\ncline{-}{i}{g}
\ncline{-}{f}{j}
\ncline{-}{j}{k}
\ncline{-}{j}{l}
\ncline{-}{zz}{g}

\cnode*(7.53,1.47){0.05}{bb}
\ncline{-}{a}{bb}
\cnode*(7.18,1.82){0.03}{r}
\cnode*(7.36,1.66){0.03}{r}
\cnode*(7.88,1.47){0.05}{ee}
\cnode*(8.18,1.77){0.05}{gg}
\cnode*(8.18,1.92){0.03}{ii}
\cnode*(8.33,1.77){0.03}{jj}
\cnode*(8.08,2.02){0.03}{ii1}
\cnode*(8.28,2.02){0.03}{ii2}
\ncline{-}{ii1}{ii}
\ncline{-}{ii2}{ii}
\ncline{-}{ii}{gg}
\ncline{-}{gg}{jj}
\ncline{-}{gg}{ee}
\cnode*(7.98,1.57){0.03}{r}
\cnode*(8.08,1.67){0.03}{r}

\cnode*(7.64,1.47){0.03}{r}
\cnode*(7.76,1.47){0.03}{r}
\ncline{-}{ee}{bb}
\cnode*(7.88,1.12){0.05}{cc}
\cnode*(7.99,1.01){0.03}{r}
\cnode*(8.11,0.9){0.03}{r}

\cnode*(8.23,0.77){0.05}{dd}
\rput(8.25,0.55){$w_3$}
\rput(4.25,0.55){$w_3$}
\ncline{-}{dd}{cc}
\cnode*(7.88,1.36){0.03}{r}
\cnode*(7.88,1.24){0.03}{r}
\ncline{-}{ee}{cc}

\cnode*(7.53,1.12){0.05}{ff}
\cnode*(7.43,1.02){0.03}{r}
\cnode*(7.33,0.92){0.03}{r}
\cnode*(7.23,0.82){0.05}{hh}
\cnode*(7.08,0.82){0.03}{kk}
\cnode*(7.23,0.67){0.03}{ll}
\ncline{-}{ff}{hh}
\ncline{-}{ff}{bb}
\ncline{-}{kk}{hh}
\ncline{-}{ll}{hh}
\cnode*(7.53,1.36){0.03}{r}
\cnode*(7.53,1.24){0.03}{r}
\ncline{-}{cc}{ff}
\cnode*(7.64,1.12){0.03}{r}
\cnode*(7.76,1.12){0.03}{r}

\cnode*(6.47,1.47){0.05}{bbb}
\cnode*(6.65,1.65){0.03}{r}
\cnode*(6.83,1.83){0.03}{r}
\ncline{-}{a}{bbb}
\cnode*(6.47,1.12){0.05}{eee}
\cnode*(6.57,1.02){0.03}{r}
\cnode*(6.67,0.92){0.03}{r}
\cnode*(6.77,0.82){0.05}{ggg}
\cnode*(6.92,0.82){0.03}{kkk}
\cnode*(6.77,0.67){0.03}{lll}
\cnode*(6.67,0.57){0.03}{lll1}
\cnode*(6.87,0.57){0.03}{lll2}
\ncline{-}{lll}{lll1}
\ncline{-}{lll}{lll2}
\ncline{-}{kkk}{ggg}
\ncline{-}{lll}{ggg}
\ncline{-}{eee}{ggg}
\cnode*(6.47,1.23){0.03}{r}
\cnode*(6.47,1.35){0.03}{r}
\ncline{-}{eee}{bbb}
\cnode*(6.12,1.47){0.05}{fff}
\cnode*(6.02,1.57){0.03}{r}
\cnode*(5.92,1.67){0.03}{r}
\cnode*(6.12,1.12){0.05}{ccc}
\cnode*(6.01,1.01){0.03}{r}
\cnode*(5.89,0.89){0.03}{r}
\cnode*(5.77,0.77){0.05}{ddd}
\rput(5.6,0.55){$w_2$}
\rput(1.6,0.55){$w_2$}
\ncline{-}{ccc}{ddd}
\cnode*(6.12,1.23){0.03}{r}
\cnode*(6.12,1.35){0.03}{r}
\cnode*(6.23,1.12){0.03}{r}
\cnode*(6.35,1.12){0.03}{r}

\ncline{-}{eee}{ccc}
\ncline{-}{ccc}{fff}
\cnode*(6.23,1.47){0.03}{r}
\cnode*(6.35,1.47){0.03}{r}

\ncline{-}{fff}{bbb}
\cnode*(5.82,1.77){0.05}{hhh}
\cnode*(5.67,1.77){0.03}{iii}
\cnode*(5.82,1.92){0.03}{jjj}
\rput(6.9,0.1){$v \in K$}
\ncline{-}{iii}{hhh}
\ncline{-}{jjj}{hhh}
\ncline{-}{fff}{hhh}
\cnode*(3,2){0.05}{a}
\rput(3,1.85){$v$}
\cnode*(3,3.5){0.05}{d}
\cnode*(4.23,0.77){0.05}{dd}
\rput(2.9,0.1){$v \in G$}
\cnode*(1.77,0.77){0.05}{ddd}
\ncline{-}{a}{d}
\ncline{-}{a}{dd}
\ncline{-}{a}{ddd}
\end{pspicture}
\end{center}
\caption{Construction of graph $K$ from $G$ for $k=4$.}\label{fig300}
\end{figure}

\begin{lemma}\label{le:main3}
Let $G$ be a cubic planar graph and let $H$ be the corresponding auxiliary graph. Let $K$ be the instance of $k$-$L(2,1)$-labelling constructed from $G$ as described above. Then $H$ has a good coloured orientation if and only if $K$ has a $k$-$L(2,1)$-labelling.
\end{lemma}

\begin{proof}
Suppose that $K$ has a $k$-$L(2,1)$-labelling $L$.
We now describe how to obtain a coloured orientation of $H$ from $L$.
Any vertex in $H$ corresponds to a vertex of degree $k-1$ in $K$ and so must be coloured $0$ or $k$. Colour a vertex of $H$ white if it corresponds to a vertex labelled $0$ in $K$ and black if it corresponds to vertex labelled $k$ in $K$.

We orient some of the edges of $H$ as follows. If $uv$ is an edge of $H$ then there is a path $u,a_u,a_v,v$ between $u,v$ in $K$ where $u$ is adjacent to $a_u$ and $v$ is adjacent to $a_v$. Orient the edge $uv$ from $u$ to $v$ if and only if $a_v \in \{0,k\}$ and orient it from $v$ to $u$ if and only if $a_u \in \{0,k\}$. From Theorem~\ref{th:colgadget} it follows that in our colouring of $H$, each vertex of $H$ is adjacent to at most one vertex of the opposite colour, and an edge is oriented if and only if it joins two vertices of the same colour.
Consider a vertex $v\in H$. All neighbours of $v$ in $K$ must receive different colours, so in $H$, $v$ has at most one incoming edge, and at most two outgoing edges. Finally let $u$ be an outvertex of $H$. Then $u$ is part of exactly one copy of the gadget $G_k$ and has a neighbour $w$ of degree $k-1$ that is not part of this copy of $G_k$. We have $\{L(u),L(w)\}=\{0,k\}$ which means that no other neighbour of $u$ is labelled $0$ or $k$ and hence $u$ has indegree $0$. Therefore $H$ has a coloured orientation and by Lemma~\ref{le:goodorient}, $H$ has a good coloured orientation.

Now suppose that $H$ has a good coloured orientation. We will show how to construct a $k$-$L(2,1)$-labelling $L$ of $K$.
First label all vertices $v$ in $K$ that appear in $H$, so that $L(v)=0$ if $v$ is coloured white in $H$ and otherwise $L(v)=k$. Next give labels to all the remaining vertices that appear in a copy of the gadget $G_k$.
Let $uv$ be an edge of $H$ and suppose without loss of generality that $L(u)=0$. Let $u,a_u,a_v,v$ be the path of length three from $u$ to $v$ in $K$. If $uv$ is not oriented, let $L(a_u)=k-1$ and $L(a_v)=1$. If $uv$ is oriented from $u$ to $v$ then let $L(a_u)=2,L(a_v)=k$ and if $uv$ is oriented from $v$ to $u$ then let $L(a_u)=k,L(a_v)=2$. Furthermore if $w \in V(H)$ then, because we start from a good coloured orientation of $H$, its three neighbours in $K$ receive different labels. Then Theorem~\ref{th:colgadget} shows that $L$ may be extended so that any vertex appearing in a copy of $G_k$ receives a label. For each outvertex, its neighbour of degree $k-1$ must be labelled. This can be done because each edge in $H$ adjacent to an outvertex $x$ is oriented away from $x$, so one of the labels $0,k$ is always available. Finally the vertices of degree one form an independent set and are all adjacent to vertices of degree $k-1$ that have received label $0$ or $k$. So they may be labelled. Hence $K$ has a $k$-$L(2,1)$-labelling.
\end{proof}

We now return to the main problem of this paper.
The following theorem which is the main statement of this paper follows immediately from Theorem~\ref{th:cubplanp} and Lemmas~\ref{le:main1}, \ref{le:goodorient},
\ref{le:main2} and~\ref{le:main3}.

\begin{theorem}
Problem~\ref{pro:main} is NP-complete.
\end{theorem}

\bibliographystyle{plain}

\end{document}